\documentclass[10pt]{amsart}
\usepackage{amssymb,MnSymbol}
\usepackage{amsthm,amsmath}
\usepackage{epic,cite}
\usepackage{ftnxtra}

\title[]{Bisymmetric And Quasitrivial Operations: Characterizations and Enumerations}

\author{Jimmy Devillet}
\address{Mathematics Research Unit, University of Luxembourg, Maison du Nombre, 6, avenue de la Fonte, L-4364 Esch-sur-Alzette, Luxembourg}
\email{jimmy.devillet[at]uni.lu}

\date{\today}

\theoremstyle{plain}
\newtheorem{theorem}{Theorem}[section] 
\newtheorem{lemma}[theorem]{Lemma}
\newtheorem{proposition}[theorem]{Proposition}
\newtheorem{corollary}[theorem]{Corollary}
\newtheorem{fact}[theorem]{Fact}

\theoremstyle{definition}
\newtheorem{definition}[theorem]{Definition}
\newtheorem{example}[theorem]{Example}

\theoremstyle{remark}

\newtheorem{remark}{Remark}

\newcommand{\cleq}{\mathrm{conv}_{\leq}}

\begin{document}
\begin{abstract}
We investigate the class of bisymmetric and quasitrivial binary operations on a given set $X$ and provide various characterizations of this class as well as the subclass of bisymmetric, quasitrivial, and order-preserving binary operations. We also determine explicitly the sizes of these classes when the set $X$ is finite.
\end{abstract}



\maketitle

\section{Introduction}

Let $X$ be an arbitrary nonempty set. We use the symbol $X_n$ if $X$ is finite of size $n\geq 1$, in which case we assume w.l.o.g.\ that $X_n=\{1,\ldots,n\}$.

Recall that a binary operation $F\colon X^2 \to X$ is \emph{bisymmetric} if it satisfies the functional equation
$$
F(F(x,y),F(u,v)) = F(F(x,u),F(y,v))
$$
for all $x,y,u,v \in X$. The bisymmetry property for binary real operations was studied by Acz\'{e}l \cite{Acz1,Acz2}. Since then, it has been investigated in the theory of functional equations, especially in characterizations of mean functions (see, e.g., \cite{Acz3,Acz4,FoMa97,GMRE09}). This property has also been extensively investigated in algebra where it is called \emph{mediality}. For instance, a groupoid $(X,F)$ where $F$ is a bisymmetric binary operation on $X$ is called a \emph{medial groupoid} (see, e.g., \cite{Jez78,Jez83,Jez83(1),Jez94,Kep81}).

In this paper, which is a continuation of \cite{CouDevMar2}, we investigate the class of binary operations $F\colon X^2\to X$ that are bisymmetric and quasitrivial, where quasitriviality means that $F$ always outputs one of its input values. It is known that any bisymmetric and quasitrivial operation is associative (see Kepka~\cite[Corollary 10.3]{Kep81}). This observation is of interest since it shows that the class of bisymmetric and quasitrivial operations is a subclass of the class of associative and quasitrivial operations. The latter class was characterized independently by several authors (see, e.g., Kepka~\cite[Corollary~1.6]{Kep81} and L\"anger~\cite[Theorem~1]{Lan80}) and a recent and elementary proof of this characterization is available in \cite[Theorem 2.1]{CouDevMar2}. We also investigate certain subclasses of bisymmetric and quasitrivial operations by adding properties such as order-preservation and existence of neutral and/or annihilator elements. In the finite case (i.e., $X=X_n$ for any integer $n\geq 1$), we enumerate these subclasses as well as the class of bisymmetric and quasitrivial operations.

The outline of this paper is as follows. After presenting in Section 2 some definitions and preliminary results (including the above-mentioned characterization of the class of associative and quasitrivial operations, see Theorem~\ref{thm:Kim}), we introduce in Section 3 the concept of ``quasilinear weak ordering'' and show how it can be used to characterize the class of bisymmetric and quasitrivial operations $F\colon X^2 \to X$ (see Theorem \ref{thm:B}). We also provide an alternative characterization of the latter class when $X=X_n$ for some integer $n \geq 1$ (see Theorem \ref{thm:B}). In particular, these characterizations provide an answer to open questions posed in \cite[Section 5, Question (b)]{CouDevMar} and \cite[Section 6]{CouDevMar2}. We then recall the weak single-peakedness property (see Definition~\ref{def:WBlack}) as a generalization of single-peakedness to arbitrary weakly ordered sets and we use it to characterize the class of bisymmetric, quasitrivial, and order-preserving operations (see Theorem \ref{thm:BND}). In Section 4, we restrict ourselves to the case where $X=X_n$ for some integer $n\geq 1$ and enumerate the class of bisymmetric and quasitrivial operations as well as some subclasses discussed in this paper. By doing so, we point out some known integer sequences and introduce new ones. In particular, the search for the number of bisymmetric and quasitrivial binary operations on $X_n$ for any integer $n\geq 1$ (see Proposition \ref{prop:qms}) gives rise to a sequence that was previously unknown in the Sloane's On-Line Encyclopedia of Integer Sequences (OEIS, see \cite{Slo}). All the sequences that we consider are given in explicit forms and through their generating functions or exponential generating functions (see Propositions \ref{prop:uw1}, \ref{prop:qms}, \ref{prop:uw2}, and \ref{prop:count}). Finally, in Section 5 we further investigate the quasilinearity property of the weak orderings on $X$ that are weakly single-peaked w.r.t.\ a fixed linear ordering on $X$ and provide a graphical characterization of these weak orderings (see Theorem \ref{thm:graph2}).

\section{Preliminaries}

In this section we recall and introduce some basic definitions and provide some preliminary results.

Recall that a binary relation $R$ on $X$ is said to be
\begin{itemize}
\item \emph{total} if $\forall x,y$: $xRy$ or $yRx$,
\item \emph{transitive} if $\forall x,y,z$: $xRy$ and $yRz$ implies $xRz$,
\item \emph{antisymmetric} if $\forall x,y$: $xRy$ and $yRx$ implies $x=y$.
\end{itemize}

A binary relation $\leq$ on $X$ is said to be a \emph{linear ordering on $X$} if it is total, transitive, and antisymmetric. In that case the pair $(X,\leq)$ is called a \emph{linearly ordered set}. For any integer $n\geq 1$, we can assume w.l.o.g.\ that the pair $(X_n,\leq_n)$ represents the set $X_n=\{1,\ldots,n\}$ endowed with the linear ordering relation $\leq_n$ defined by $1<_n\cdots <_nn$.

A binary relation $\lesssim$ on $X$ is said to be a \emph{weak ordering on $X$} if it is total and transitive. In that case the pair $(X,\lesssim)$ is called a \emph{weakly ordered set}. We denote the symmetric and asymmetric parts of $\lesssim$ by $\sim$ and $<$, respectively. Recall that $\sim$ is an equivalence relation on $X$ and that $<$ induces a linear ordering on the quotient set ${X/\sim}$. For any $u\in X$ we denote the equivalence class of $u$ by $[u]_{\sim}$, i.e., $[u]_{\sim}=\{x\in X\mid x\sim u\}$.

For any linear ordering $\preceq$ and any weak ordering $\lesssim$ on $X$, we say that $\preceq$ is \emph{subordinated to} $\lesssim$ if for any $x,y \in X$, we have that $x < y$ implies $x \prec y$.

For a weak ordering $\lesssim$ on $X$, an element $u\in X$ is said to be \emph{maximal} (resp.\ \emph{minimal}) \emph{for $\lesssim$} if $x\lesssim u$ (resp.\ $u\lesssim x$) for all $x\in X$. We denote the set of maximal (resp.\ minimal) elements of $X$ for $\lesssim$ by $\max_{\lesssim}X$ (resp.\ $\min_{\lesssim}X$).

\begin{definition}\label{de:def1}
An operation $F\colon X^2\to X$ is said to be
\begin{itemize}
\item \emph{associative} if $F(F(x,y),z) = F(x,F(y,z))$ for all $x,y,z\in X$,
\item \emph{bisymmetric} if $F(F(x,y),F(u,v)) = F(F(x,u),F(y,v))$ for all $x,y,u,v \in X$,
\item \emph{idempotent} if $F(x,x)=x$ for all $x\in X$,
\item \emph{quasitrivial} (or \emph{conservative}) if $F(x,y)\in\{x,y\}$ for all $x,y\in X$,
\item \emph{commutative} if $F(x,y)=F(y,x)$ for all $x,y\in X$,
\item \emph{$\leq$-preserving} for some linear ordering $\leq$ on $X$ if for any $x,y,x',y'\in X$ such that $x\leq x'$ and $y\leq y'$ we have $F(x,y)\leq F(x',y')$.
\end{itemize}
\end{definition}

\begin{definition}
Let $F\colon X^2\to X$ be an operation.
\begin{itemize}
\item An element $e\in X$ is said to be a \emph{neutral element of $F$} if $F(e,x) = F(x,e) = x$ for all $x\in X$. In this case we can easily show by contradiction that the neutral element is unique.
\item An element $a\in X$ is said to be an \emph{annihilator of $F$} if $F(x,a) = F(a,x) = a$ for all $x\in X$. In this case we can easily show by contradiction that the annihilator is unique.
\item The points $(x,y)$ and $(u,v)$ of $X^2$ are said to be \emph{$F$-connected} if $F(x,y) = F(u,v)$. The point $(x,y)$ of $X^2$ is said to be \emph{$F$-isolated} if it is not $F$-connected to another point of $X^2$.
\end{itemize}
\end{definition}


When $X_n$ is endowed with $\leq_n$, the operations $F\colon X_n^2\to X_n$ can be easily visualized by showing their contour plots, where we connect by edges or paths all the points of $X_n^2$ having the same values by $F$. For instance, the operation $F\colon X_3^2\to X_3$ defined by $F(2,2)=F(2,3)=2$, $F(1,x)=1$, and $F(3,x)=F(2,1)=3$ for $x = 1,2,3$, is idempotent. Its contour plot is shown in Figure~\ref{fig:fs}.

\setlength{\unitlength}{5.0ex}
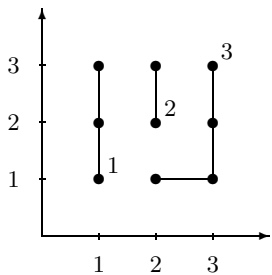
\begin{figure}[htbp]
\begin{center}
\begin{small}
\begin{picture}(5,5)
\put(0.5,0.5){\vector(1,0){4}}\put(0.5,0.5){\vector(0,1){4}}
\multiput(1.5,0.45)(1,0){3}{\line(0,1){0.1}}%
\multiput(0.45,1.5)(0,1){3}{\line(1,0){0.1}}%
\put(1.5,0){\makebox(0,0){$1$}}\put(2.5,0){\makebox(0,0){$2$}}\put(3.5,0){\makebox(0,0){$3$}}
\put(0,1.5){\makebox(0,0){$1$}}\put(0,2.5){\makebox(0,0){$2$}}\put(0,3.5){\makebox(0,0){$3$}}
\multiput(1.5,1.5)(0,1){3}{\multiput(0,0)(1,0){3}{\circle*{0.2}}}
\drawline[1](1.5,1.5)(1.5,3.5)\drawline[1](2.5,2.5)(2.5,3.5)\drawline[1](3.5,1.5)(2.5,1.5)\drawline[1](3.5,1.5)(3.5,3.5)
\put(1.75,1.75){\makebox(0,0){$1$}}\put(2.75,2.75){\makebox(0,0){$2$}}\put(3.75,3.75){\makebox(0,0){$3$}}
\end{picture}
\end{small}
\caption{An idempotent operation on $X_3$ (contour plot)}
\label{fig:fs}
\end{center}
\end{figure}

\begin{definition}
Let $\leq$ be a linear ordering on $X$. We say that an operation $F \colon X^2 \to X$ has
\begin{itemize}
\item a \emph{$\leq$-disconnected level set} if there exist $x,y,u,v,s,t \in X$, with $(x,y) < (u,v) < (s,t)$, such that $F(x,y) = F(s,t) \neq F(u,v)$.
\item a \emph{horizontal (resp.\ vertical) $\leq$-disconnected level set} if there exist $x,y,z,u \in X$, with $x < y < z$, such that $F(x,u) = F(z,u) \neq F(y,u)$ (resp.\ $F(u,x) = F(u,z) \neq F(u,y)$).
\end{itemize}
\end{definition}

\begin{fact}\label{fact:dis}
Let $\leq$ be a linear ordering on $X$. If $F \colon X^2 \to X$ has a horizontal or vertical $\leq$-disconnected level set, then it has a $\leq$-disconnected level set.
\end{fact}

\begin{remark}
We observe that, for any linear ordering $\leq$ on $X$, an operation $F \colon X^2 \to X$ having a $\leq$-disconnected level set need not have a horizontal or vertical $\leq$-disconnected level set. Indeed, the operation $F\colon X_3^2\to X_3$ whose contour plot is depicted in Figure \ref{fig:jump} has a $\leq_3$-disconnected level set since $F(1,1) = F(2,3) = 1 \neq 2 = F(2,2)$ but it has no horizontal or vertical $\leq_3$-disconnected level set.
\end{remark}

\setlength{\unitlength}{5.0ex}
\begin{figure}[htbp]
\begin{center}
\begin{small}
\begin{picture}(3,3)
\multiput(0.5,0.5)(0,1){3}{\multiput(0,0)(1,0){3}{\circle*{0.2}}}
\drawline[1](0.5,2.5)(0.5,1.5)\drawline[1](0.5,1.5)(1.5,1.5)\drawline[1](1.5,1.5)(1.5,0.5)\drawline[1](1.5,0.5)(2.5,0.5)\drawline[1](0.5,0.5)(1.5,2.5)\drawline[1](2.5,1.5)(2.5,2.5)
\put(0.75,0.5){\makebox(0,0){$1$}}\put(1.75,1.75){\makebox(0,0){$2$}}\put(2.75,2.75){\makebox(0,0){$3$}}
\end{picture}
\end{small}
\caption{An idempotent operation on $X_3$}
\label{fig:jump}
\end{center}
\end{figure}
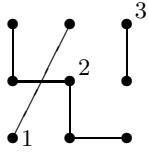

\begin{lemma}\label{lem:dis}
Let $\leq$ be a linear ordering on $X$. If $F \colon X^2 \to X$ is quasitrivial, then it has a $\leq$-disconnected level set iff it has a horizontal or vertical $\leq$-disconnected level set.
\end{lemma}

\begin{proof}
(Necessity) Suppose that $F$ has a $\leq$-disconnected level set and let us show that it has a horizontal or vertical $\leq$-disconnected level set. By assumption, there exist $x,y,u,v,s,t \in X$, with $(x,y) < (u,v) < (s,t)$, such that $F(x,y) = F(s,t) \neq F(u,v)$. Since $F$ is quasitrivial, we have $F(x,y) \in \{x,y\}$. Suppose that $F(x,y)=x$ (the other case is similar). Also, since $F$ is quasitrivial, we have $s=x$ or $t=x$. If $s=x$, then $u = x$ and thus $F$ has a vertical $\leq$-disconnected level set. Otherwise, if $t=x$ and $s \neq x$, then $y \leq x$. If $y = x$, then $v=x$ and thus $F$ has a horizontal $\leq$-disconnected level set. Otherwise, if $y < x$, then considering the point $(s,y) \in X^2$, we get $(x,y) < (s,y) < (s,x)$ and $F(x,y) = F(s,x) = x \neq F(s,y) \in \{s,y\} = \{F(s,s),F(y,y)\}$, which shows that $F$ has either a horizontal or a vertical $\leq$-disconnected level set.

(Sufficiency) This follows from Fact \ref{fact:dis}.
\end{proof}

We define the \emph{strict convex hull of $x,y \in X$ w.r.t.\ $\leq$} by $\mathrm{conv}_{\leq}(x,y) = \{z \in X \mid x < z < y\}$, if $x<y$, and $\mathrm{conv}_{\leq}(x,y) = \{z \in X \mid y < z < x\}$, if $y<x$.

Recall that for any linear ordering $\leq$ on $X$, a subset $C$ of $X$ is said to be \emph{convex w.r.t.\ $\leq$} if for any $x,y,z\in X$ such that $y\in\cleq(x,z)$, we have that $x,z\in C$ implies $y\in C$.

\begin{remark}
\begin{enumerate}
\item[(a)]For any quasitrivial operation $F\colon X^2 \to X$ and any $x \in X$, consider the sets
$$
L_x^h(F) = \{y \in X \mid F(y,x) = x\} \quad\text{and}\quad L_x^v(F) = \{y \in X \mid F(x,y) = x\}.
$$
Clearly, for any linear ordering $\leq$ on $X$, a quasitrivial operation $F\colon X^2 \to X$ has no $\leq$-disconnected level set iff for any $x \in X$, the sets $L_x^h(F)$ and $L_x^v(F)$ are convex w.r.t.\ $\leq$.
\item[(b)] Recall that the \emph{kernel of an operation $F\colon X^2 \to X$} is the set
$$
\ker(F) = \{((a,b),(c,d)) \in X^2 \times X^2 \mid F(a,b) = F(c,d)\}.
$$
Clearly, $\ker(F)$ is an equivalence relation on $X^2$. It is not difficult to see that for any linear ordering $\leq$ on $X$, a quasitrivial operation $F\colon X^2 \to X$ has no $\leq$-disconnected level set iff for any $x \in X$, the class of $(x,x)$ w.r.t. $\ker(F)$ is convex w.r.t. $\leq$.
\end{enumerate}
\end{remark}

\begin{fact}\label{lem:jp}
Let $\leq$ be a linear ordering on $X$. If $F\colon X^2 \to X$ is $\leq$-preserving then it has no $\leq$-disconnected level set.
\end{fact}

\begin{proposition}\label{prop:jump}
Let $\leq$ be a linear ordering on $X$. If $F\colon X^2 \to X$ is quasitrivial, then it is $\leq$-preserving iff it has no $\leq$-disconnected level set.
\end{proposition}

\begin{proof}
(Necessity) This follows from Fact \ref{lem:jp}.

(Sufficiency) Suppose that $F$ has no $\leq$-disconnected level set and let us show by contradiction that $F$ is $\leq$-preserving. Suppose for instance that there exist $x,y,z \in X$, $y < z$, such that $F(x,y) > F(x,z)$. By quasitriviality we see that $x \notin \{y,z\}$. Suppose for instance that $x<y<z$ (the other cases are similar). By quasitriviality we have $F(x,y) = y$ and $F(x,z)=x=F(x,x)$, and hence by Lemma \ref{lem:dis}, $F$ has a $\leq$-disconnected level set, a contradiction.
\end{proof}

\begin{remark}
We cannot relax quasitriviality into idempotency in Proposition \ref{prop:jump}. Indeed, the operation $F\colon X_3^2\to X_3$ whose contour plot is depicted in Figure~\ref{fig:fs} is idempotent and has no $\leq_3$-disconnected level set. However it is not $\leq_3$-preserving.
\end{remark}

Certain links between associativity and bisymmetry were investigated by several authors (see, e.g., \cite{CouDevMar,Kep81,MasMonTor03,SuLiuPed17,San05}). We gather them in the following lemma.

\begin{lemma}[{see, \cite[Lemma 22]{CouDevMar}}]\label{lemma:bis}
\begin{enumerate}
\item[(i)] If $F\colon X^2 \to X$ is bisymmetric and has a neutral element, then it is associative and commutative.
\item[(ii)] If $F\colon X^2 \to X$ is associative and commutative, then it is bisymmetric.
\item[(iii)] If $F\colon X^2 \to X$ is quasitrivial and bisymmetric, then it is associative.
\end{enumerate}
\end{lemma}

\begin{corollary}\label{cor:equiv}
\begin{enumerate}
\item[(i)] If $F\colon X^2 \to X$ is commutative and quasitrivial, then it is  associative iff it is bisymmetric.
\item[(ii)] If $F\colon X^2 \to X$ has a neutral element, then it is associative and commutative iff it is bisymmetric.
\end{enumerate}
\end{corollary}

\begin{remark}\label{rem:sym}
\begin{enumerate}
\item[(a)] In \cite[Theorem 3.3]{CouDevMar2}, the class of associative, commutative, and quasitrivial operations was characterized. In particular, it was shown that there are exactly $n!$ associative, commutative, and quasitrivial operations on $X_n$. Using Corollary \ref{cor:equiv}(i) we can replace associativity with bisymmetry in \cite[Theorem 3.3]{CouDevMar2}.
\item[(b)] In \cite[Theorem 3.7]{CouDevMar2}, the class of associative, commutative, quasitrivial, and order-preserving operations was characterized. In particular, it was shown that there are exactly $2^{n-1}$ associative, commutative, quasitrivial, and $\leq_n$-preserving operations on $X_n$. Using Corollary \ref{cor:equiv}(i) we can replace associativity with bisymmetry in \cite[Theorem 3.7]{CouDevMar2}.
\end{enumerate}
\end{remark}

\begin{lemma}[{see \cite[Proposition~4]{CouDevMar}}]\label{lem:iso}
Let $F\colon X^2\to X$ be a quasitrivial operation and let $e\in X$. Then $e$ is a neutral element of $F$ iff $(e,e)$ is $F$-isolated.
\end{lemma}

For any integer $n\geq 1$, any $F\colon X_n^2\to X_n$, and any $z\in X_n$, the \emph{F-degree of $z$}, denoted $\deg_F(z)$, is the number of points $(x,y)\in X_n^2\setminus\{(z,z)\}$ such that $F(x,y)=F(z,z)$. Also, the \emph{degree sequence of $F$}, denoted $\deg_F$, is the nondecreasing $n$-element sequence of the degrees $\deg_F(x)$, $x\in X_n$. For instance, the degree sequence of the operation $F\colon X_3^2 \to X_3$ defined in Figure \ref{fig:fs} is $\deg_F = (1,2,3)$.

\begin{remark}
If $F\colon X_n^2\to X_n$ is a quasitrivial operation, then there exists at most one element $x \in X_n$ such that $\deg_F(x) = 0$. Indeed, otherwise by Lemma \ref{lem:iso}, $F$ would have at least two distinct neutral elements, a contradiction.
\end{remark}

\begin{lemma}\label{lem:deg}
If $F\colon X_n^2\to X_n$ is quasitrivial, then for all $x \in X_n$, we have $\deg_F(x) \leq 2(n-1)$.
\end{lemma}

\begin{proof}
Let $x \in X_n$. Since $F\colon X_n^2\to X_n$ is quasitrivial, the point $(x,x)$ can only be $F$-connected to points of the form $(x,y)$ or $(y,x)$ for some $y \in X_n\setminus \{x\}$.
\end{proof}

The following result was mentioned in \cite[Section 2]{CouDevMar2} without proof.

\begin{proposition}\label{prop:an}
Let $F\colon X_n^2\to X_n$ be a quasitrivial operation and let $a \in X_n$. Then $a$ is an annihilator of $F$ iff $\deg_F(a) = 2(n-1)$.
\end{proposition}

\begin{proof}
(Necessity) By definition of an annihilator, we have $F(x,a) = F(a,x) = a = F(a,a)$ for all $x \in X_{n}\setminus\{a\}$. Thus, we have $\deg_F(a) \geq 2(n-1)$ and hence by Lemma \ref{lem:deg} we conclude that $\deg_F(a) = 2(n-1)$.

(Sufficiency) By quasitriviality, the point $(a,a)$ cannot be $F$-connected to a point $(u,v)$ where $u \neq a$ and $v \neq a$. Thus, since $\deg_F(a) = 2(n-1)$, we have $a=F(a,a)=F(x,a)=F(a,x)$ for all $x \in X_n \setminus \{a\}$.
\end{proof}

\begin{remark}
We observe that Proposition \ref{prop:an} no longer holds if we relax quasitriviality into idempotency. Indeed, the operation $F\colon X_3^2 \to X_3$ whose contour plot is depicted in Figure \ref{fig:an} is idempotent and the element $a=1$ is the annihilator of $F$. However, $\deg_F(1) = 6 > 4$.
\end{remark}

\setlength{\unitlength}{5.0ex}
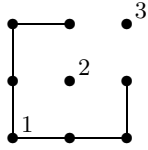
\begin{figure}[htbp]
\begin{center}
\begin{small}
\begin{picture}(3,3)
\multiput(0.5,0.5)(0,1){3}{\multiput(0,0)(1,0){3}{\circle*{0.2}}}
\drawline[1](0.5,0.5)(0.5,2.5)\drawline[1](0.5,0.5)(2.5,0.5)\drawline[1](2.5,0.5)(2.5,1.5)\drawline[1](0.5,2.5)(1.5,2.5)
\put(0.75,0.75){\makebox(0,0){$1$}}\put(1.75,1.75){\makebox(0,0){$2$}}\put(2.75,2.75){\makebox(0,0){$3$}}
\end{picture}
\end{small}
\caption{An idempotent operation with an annihilator on $X_3$}
\label{fig:an}
\end{center}
\end{figure}


The \emph{projection operations} $\pi_1\colon X^2\to X$ and $\pi_2\colon X^2\to X$ are respectively defined by $\pi_1(x,y)=x$ and $\pi_2(x,y)=y$ for all $x,y\in X$.

Given a weak ordering $\lesssim$ on $X$, the \emph{maximum} (resp.\ \emph{minimum}) \emph{operation on $X$ w.r.t.\ $\lesssim$} is the commutative binary operation $\max_{\lesssim}$ (resp.\ $\min_{\lesssim}$) defined on $X^2\setminus\{(x,y)\in X^2\mid x\sim y,~x\neq y\}$ as $\max_{\lesssim}(x,y)=y$ (resp.\ $\min_{\lesssim}(x,y)=x$) whenever $x\lesssim y$. We also note that if $\lesssim$ reduces to a linear ordering, then the operation $\max_{\lesssim}$ (resp.\ $\min_{\lesssim}$) is defined everywhere on $X^2$.

The following theorem provides a characterization of the class of associative and quasitrivial operations on $X$. In \cite[Section~1.2]{Ack}, it was observed that this result is a simple consequence of two papers on idempotent semigroups (see ~\cite{Kim} and ~\cite{Dave}). It was also independently discovered by various authors (see, e.g.,~\cite[Corollary~1.6]{Kep81} and ~\cite[Theorem~1]{Lan80}). A short and elementary proof was also given in \cite[Theorem 2.1]{CouDevMar2}.

\begin{theorem}\label{thm:Kim}
An operation $F\colon X^2\to X$ is associative and quasitrivial iff there exists a weak ordering $\precsim$ on $X$ such that
\begin{equation}\label{eq:kimura}
F|_{A\times B} ~=~
\begin{cases}
\max_{\precsim}|_{A\times B}, & \text{if $A\neq B$},\\
\pi_1|_{A\times B}\hspace{1.5ex}\text{or}\hspace{1.5ex}\pi_2|_{A\times B}, & \text{if $A=B$},
\end{cases}
\qquad \forall A,B\in {X/\sim}.
\end{equation}
\end{theorem}

\begin{remark}\label{rem:con}
As observed in \cite[Section 2]{CouDevMar2}, for any associative and quasitrivial operation $F\colon X^2 \to X$, there is exactly one weak ordering $\precsim$ on $X$ for which $F$ is of the form \eqref{eq:kimura}. This weak ordering is defined by: $x \precsim y$ iff $F(x,y)=y$ or $F(y,x)=y$. Moreover, as observed in \cite[Corollary 2.3]{CouDevMar2}, if $X=X_n$ for some integer $n \geq 1$, then $\precsim$ can be defined as follows: $x \precsim y$ ~iff~ $\deg_F(x) \leq \deg_F(y)$. The latter observation follows from \cite[Proposition 2.2]{CouDevMar2} which states that for any $x \in X_n$ we have
\begin{equation}\label{eq:degx}
\deg_F(x) ~=~ 2\times|\{z\in X_n\mid z\prec x\}|+|\{z\in X_n\mid z\sim x,~z\neq x\}|.
\end{equation}
\end{remark}

Some associative and quasitrivial operations are bisymmetric. For instance, so are the projection operations $\pi_1$ and $\pi_2$. However, some other operations are not bisymmetric. Indeed, the operation $F\colon X_3^2 \to X_3$ whose contour plot is depicted in Figure \ref{fig:kim} is associative and quasitrivial since it is of the form \eqref{eq:kimura} for the weak ordering $\precsim$ on $X_3$ defined by $2 \prec 1 \sim 3$. However, this operation is not bisymmetric since $F(F(2,3),F(1,2)) = F(3,1) = 3 \neq 1 = F(1,3) = F(F(2,1),F(3,2))$. In the next section, we provide a characterization of the subclass of bisymmetric and quasitrivial operations (see Theorem \ref{thm:B}).

\setlength{\unitlength}{5.0ex}
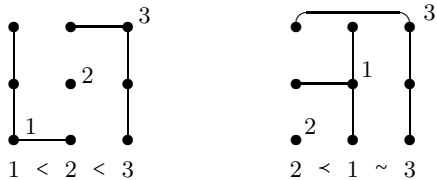
\begin{figure}[tbp]
\begin{center}
\begin{small}
\begin{picture}(3,3)
\multiput(0.5,0.5)(0,1){3}{\multiput(0,0)(1,0){3}{\circle*{0.2}}}
\drawline[1](0.5,0.5)(0.5,2.5)\drawline[1](0.5,0.5)(1.5,0.5)
\drawline[1](2.5,2.5)(2.5,0.5)\drawline[1](2.5,2.5)(1.5,2.5)
\put(0.8,0.75){\makebox(0,0){$1$}}
\put(1.8,1.65){\makebox(0,0){$2$}}
\put(2.8,2.7){\makebox(0,0){$3$}}
\put(0.5,0){\makebox(0,0){$1$}}\put(1,0){\makebox(0,0){$<$}}\put(1.5,0){\makebox(0,0){$2$}}
\put(2,0){\makebox(0,0){$<$}}\put(2.5,0){\makebox(0,0){$3$}}
\end{picture}
\hspace{0.10\textwidth}
\begin{picture}(3,3)
\multiput(0.5,0.5)(0,1){3}{\multiput(0,0)(1,0){3}{\circle*{0.2}}}
\drawline[1](1.5,2.5)(1.5,0.5)\drawline[1](1.5,1.5)(0.5,1.5)\drawline[1](2.5,2.5)(2.5,0.5) \put(1.5,2.5){\oval(2,0.5)[t]}
\put(0.75,0.75){\makebox(0,0){$2$}}
\put(1.75,1.75){\makebox(0,0){$1$}}
\put(2.85,2.75){\makebox(0,0){$3$}}
\put(0.5,0){\makebox(0,0){$2$}}\put(1,0){\makebox(0,0){$\prec$}}\put(1.5,0){\makebox(0,0){$1$}}
\put(2,0){\makebox(0,0){$\sim$}}\put(2.5,0){\makebox(0,0){$3$}}
\end{picture}
\end{small}
\caption{An associative and quasitrivial operation on $X_3$ that is not bisymmetric}
\label{fig:kim}
\end{center}
\end{figure}

%

\section{Characterizations of bisymmetric and quasitrivial operations}

In this section we provide characterizations of the class of bisymmetric and quasitrivial operations $F\colon X^2 \to X$ as well as the subclass of bisymmetric, quasitrivial, and $\leq$-preserving operations $F\colon X^2\to X$ for some fixed linear ordering $\leq$ on $X$.

\begin{definition}\label{def:quasi}
We say that a weak ordering $\precsim$ on $X$ is \emph{quasilinear} if there exist no pairwise distinct $a,b,c\in X$ such that $a \prec b \sim c$.
\end{definition}

Thus defined, a weak ordering $\precsim$ on $X$ is quasilinear iff for every $x,y \in X$, $x\neq y$, we have that $x \sim y$ implies $x,y \in \min_{\precsim}X$.

\begin{remark}
For any integer $n \geq 1$, the weak orderings $\precsim$ on $X_n$ that are quasilinear are known in social choice theory as \emph{top orders} (see, e.g., \cite[Section 2]{Fi15}).
\end{remark}

\begin{fact}\label{lem:count}
Let $\precsim$ be a quasilinear weak ordering on $X$.
\begin{itemize}
\item If $|\min_{\precsim}X| = 1$ then it is a linear ordering.
\item If $\max_{\precsim}X \neq \emptyset$, then $\max_{\precsim}X = X$ or $|\max_{\precsim}X| = 1$.
\end{itemize}
\end{fact}

\begin{proposition}\label{prop:jim}
Let $F \colon X^2 \to X$ be of the form \eqref{eq:kimura} for some weak ordering $\precsim$ on $X$. Then, $\precsim$ is quasilinear iff for any linear ordering $\preceq'$ on $X$ subordinated to $\precsim$, $F$ is $\preceq'$-preserving.
\end{proposition}

\begin{proof}
(Necessity) By the form \eqref{eq:kimura} of $F$ and by quasilinearity of $\precsim$, $F$ is $\preceq'$-preserving for any linear ordering $\preceq'$ on $X$ subordinated to $\precsim$.

(Sufficiency) We proceed by contradiction. Suppose that $F$ is $\preceq'$-preserving for any linear ordering $\preceq'$ on $X$ subordinated to $\precsim$. Suppose also that there exist pairwise distinct $a,b,c\in X$, such that $a \prec b \sim c$. Fix a linear ordering $\preceq'$ on $X$ subordinated to $\precsim$. Suppose that $a \prec' b \prec' c$ (the other case is similar). If $F|_{[b]_{\sim}^2} = \pi_{1}|_{[b]_{\sim}^2}$, then
$$
F(b,c) = b \prec' c = F(a,c),
$$
a contradiction. The case where $F|_{[b]_{\sim}^2} = \pi_{2}|_{[b]_{\sim}^2}$ is similar.
\end{proof}

\begin{proposition}\label{prop:BS}
Let $F \colon X^2 \to X$ be of the form \eqref{eq:kimura} for some weak ordering $\precsim$ on $X$. Then, $F$ is bisymmetric iff $\precsim$ is quasilinear.
\end{proposition}

\begin{proof}
(Necessity) We proceed by contradiction. Suppose that $F$ is bisymmetric and suppose also that there exist pairwise distinct $a,b,c \in X$ such that $a \prec b \sim c$.  If $F|_{[b]_{\sim}^2} = \pi_{1}|_{[b]_{\sim}^2}$ then
$$
F(F(a,c),F(b,a)) = F(c,b) = c \neq b = F(b,c) = F(F(a,b),F(c,a)),
$$
a contradiction. The case where $F|_{[b]_{\sim}^2} = \pi_{2}|_{[b]_{\sim}^2}$ is similar.

(Sufficiency) Let $x,y,u,v \in X$, not all equal, and let us show that
$$F(F(x,y),F(u,v)) = F(F(x,u),F(y,v)).$$
By applying Fact \ref{lem:count} to the subset $\{x,y,u,v\}$ of $X$ we have two cases to consider.
\begin{itemize}
\item If $x \sim y \sim u \sim v$ then $F|_{\{x,y,u,v\}} = \pi_i|_{\{x,y,u,v\}}$, $i \in \{1,2\}$, and hence $F|_{\{x,y,u,v\}}$ is bisymmetric.
\item Otherwise, if $|\max_{\precsim}{\{x,y,u,v\}}| = 1$ then we can assume w.l.o.g. that $x \precsim y \precsim u \prec v$.
By the form \eqref{eq:kimura} of $F$ it is clear that $v$ is the annihilator of $F|_{\{x,y,u,v\}}$. Hence, we have
$$
F(F(x,y),F(u,v)) = v = F(F(x,u),F(y,v)).
$$
\end{itemize}
Therefore, $F$ is bisymmetric.
\end{proof}

The following proposition provides an additional characterization of the class of bisymmetric and quasitrivial operations when $X=X_n$ for some integer $n \geq 1$.

\begin{proposition}\label{thm:contour}
Let $F \colon X_n^2 \to X_n$ be quasitrivial. Then, $F$ is bisymmetric iff it is associative and satisfies
\begin{equation}\label{eq:deg}
\deg_F = (\underbrace{k-1,\ldots,k-1}_{\textstyle{k~ \text{times}}},2k,2k+2,\ldots,2n-2)
\end{equation}
for some $k \in \{1,\ldots,n\}$.
\end{proposition}

\begin{proof}
(Necessity) By Lemma \ref{lemma:bis}(iii), $F$ is associative. By Proposition \ref{prop:BS}, there exists a quasilinear weak ordering $\precsim$ on $X_n$ such that $F$ is of the form \eqref{eq:kimura}. By Remark \ref{rem:con}, we see that \eqref{eq:degx} holds for any $x\in X_n$. Therefore, we obtain the claimed form of the degree sequence of $F$.

(Sufficiency) By Theorem \ref{thm:Kim}, there exists a weak ordering $\precsim$ on $X_n$ such that $F$ is of the form \eqref{eq:kimura}. By Remark \ref{rem:con}, we have that $x \precsim y$ iff $\deg_F(x) \leq \deg_F(y)$. Thus, according to our assumptions on the degree sequence of $F$, we have that $\precsim$ is quasilinear. Finally, using Proposition \ref{prop:BS}, $F$ is bisymmetric.
\end{proof}

The following theorem, which is an immediate consequence of Lemma \ref{lemma:bis}(iii), Theorem \ref{thm:Kim}, and Propositions \ref{prop:jump}, \ref{prop:jim}, \ref{prop:BS}, and \ref{thm:contour}, provides characterizations of the class of bisymmetric and quasitrivial operations.

\begin{theorem}\label{thm:B}
Let $F \colon X^2 \to X$ be an operation. The following assertions are equivalent.
\begin{enumerate}
  \item[(i)] $F$ is bisymmetric and quasitrivial.
  \item[(ii)] $F$ is of the form \eqref{eq:kimura} for some quasilinear weak ordering $\precsim$ on $X$.
  \item[(iii)] $F$ is of the form \eqref{eq:kimura} for some weak ordering $\precsim$ on $X$ and for any linear ordering $\preceq'$ subordinated to $\precsim$, $F$ is $\preceq'$-preserving.
  \item[(iv)] $F$ is of the form \eqref{eq:kimura} for some weak ordering $\precsim$ on $X$ and for any linear ordering $\preceq'$ subordinated to $\precsim$, $F$ has no $\preceq'$-disconnected level set.
\end{enumerate}
If $X=X_n$ for some integer $n \geq 1$, then any of the assertions (i)-(iv) is equivalent to the following one.
\begin{enumerate}
\item[(v)] $F$ is associative, quasitrivial, and satisfies \eqref{eq:deg} for some $k \in \{1,\ldots,n\}$.
\end{enumerate}
\end{theorem}

\begin{remark}
\begin{enumerate}
\item[(a)] We observe that an alternative characterization of the class of bisymmetric and quasitrivial operations $F\colon X^2 \to X$ was obtained in \cite[Proposition 10.2]{Kep81}.
\item[(b)] We observe that in \cite[Theorem 3.3]{CouDevMar2} the authors proved that an operation $F\colon X_n^2 \to X_n$ is associative, commutative, and quasitrivial iff it is quasitrivial and satisfies $\deg_F = (0,2,\ldots,2(n-1))$. Surprisingly, in Theorem \ref{thm:B}(v), we have obtained a similar result by relaxing commutativity into bisymmetry. Moreover, it provides an easy test to check whether an associative and quasitrivial operation on $X_n$ is bisymmetric. Indeed, the associative and quasitrivial operation $F\colon X_3^2 \to X_3$ whose contour plot is depicted in Figure \ref{fig:kim} is not bisymmetric since $\deg_F = (0,3,3)$ is not of the form given in Theorem \ref{thm:B}(v).
\end{enumerate}
\end{remark}

The rest of this section is devoted to the subclass of bisymmetric, quasitrivial, and $\leq$-preserving operations $F\colon X^2 \to X$ for a fixed linear ordering $\leq$ on $X$. In order to characterize this subclass, we first need to recall the concept of weak single-peakedness.

\begin{definition}{(see \cite{CouDevMar2,DevKiMar17})}\label{def:WBlack}
Let $\leq$ be a linear ordering on $X$ and let $\precsim$ be a weak ordering on $X$. The weak ordering $\precsim$ is said to be \emph{weakly single-peaked with respect to $\leq$} if for any $x,y,z\in X$ such that $y \in \mathrm{conv}_{\leq}(x,z)$, we have $y \prec x$ or $y \prec z$ or $x \sim y \sim z$. If the weak ordering $\precsim$ is a linear ordering, then it is said to be \emph{single-peaked with respect to $\leq$}.
\end{definition}

\begin{remark}\label{def:Black}
Note that single-peakedness was first introduced by Black \cite{Bla48} for linear orderings on finite sets. It is also easy to show by induction that there are exactly $2^{n-1}$ single-peaked linear orderings on $X_n$ w.r.t.\ $\leq_n$ (see, e.g., \cite{BerPer06}).
\end{remark}

\begin{fact}\label{lem:min}
Let $\leq$ be a linear ordering on $X$ and let $\precsim$ be a quasilinear weak ordering on $X$ that is weakly single-peaked w.r.t.\ $\leq$.
If $|\min_{\precsim}X| = 1$, then $\precsim$ is a linear ordering that is single-peaked w.r.t.\ $\leq$.
\end{fact}

The next proposition provides a characterization of the class of associative, quasitrivial, and $\leq$-preserving operations $F\colon X^2 \to X$ for a fixed linear ordering $\leq$ on $X$.

\begin{proposition}{(see \cite[Theorem 4.5]{CouDevMar2})}\label{thm:JL}
Let $\leq$ be a linear ordering on $X$. An operation $F\colon X^2\to X$ is associative, quasitrivial, and $\leq$-preserving iff $F$ is of the form \eqref{eq:kimura} for some weak ordering $\precsim$ on $X$ that is weakly single-peaked w.r.t.\ $\leq$.
\end{proposition}

Using Lemma \ref{lemma:bis}(iii), Theorem \ref{thm:B}, and Proposition \ref{thm:JL} we can easily derive the following result.

\begin{proposition}\label{thm:bs}
Let $\leq$ be a linear ordering on $X$. An operation $F \colon X^2 \to X$ is bisymmetric, quasitrivial, and $\leq$-preserving iff $F$ is of the form \eqref{eq:kimura} for some quasilinear weak ordering $\precsim$ on $X$ that is weakly single-peaked w.r.t.\ $\leq$.
\end{proposition}

When $X=X_n$ for some integer $n\geq 1$, we provide in Proposition \ref{thm:contour2} an additional characterization of the class of bisymmetric, quasitrivial, and $\leq_n$-preserving operations. We first consider two preliminary results.

\begin{proposition}\label{prop:pi}
An operation $F \colon X_n^2 \to X_n$ is quasitrivial, $\leq_n$-preserving, and satisfies $\deg_F = (n-1,\ldots,n-1)$ iff $F=\pi_1$ or $F=\pi_2$.
\end{proposition}

\begin{proof}
(Necessity) Since $F$ is quasitrivial we know that $F(1,n) \in \{1,n\}$. Suppose that $F(1,n) = n = F(n,n)$ (the other case is similar). Since $F$ is $\leq_n$-preserving, we have $F(x,n)=n$ for all $x \in X_n$. Since $\deg_F(n) = n-1$, it follows that $F(n,y) = y$ for all $y \in X_n$. In particular, we have $F(n,1) = 1 = F(1,1)$, and by $\leq_n$-preservation we obtain $F(x,1) = 1$ for all $x \in X_n$. Finally, since $\deg_F(1) = n-1$, it follows that $F(1,y) = y$ for all $y \in X_n$. Thus, since $F$ is $\leq_n$-preserving, we have
$$
y = F(1,y) \leq_n F(x,y) \leq_n F(n,y) = y, \quad x,y \in X_n,
$$
which shows that $F = \pi_2$.

(Sufficiency) Obvious.
\end{proof}

\begin{remark}
We observe that Proposition \ref{prop:pi} no longer holds if quasitriviality is relaxed into idempotency. Indeed, the operation $F\colon X_3^2 \to X_3$ whose contour plot is depicted in Figure \ref{fig:2} is idempotent, $\leq_3$-preserving, and satisfies $\deg_F=(2,2,2)$ but it is neither $\pi_1$ nor $\pi_2$. We also observe that Proposition \ref{prop:pi} no longer holds if we omit $\leq_n$-preservation. Indeed, the operation $F\colon X_3^2 \to X_3$ whose contour plot is depicted in Figure \ref{fig:3} is quasitrivial and satisfies $\deg_F=(2,2,2)$ but it is neither $\pi_1$ nor $\pi_2$.
\end{remark}

\setlength{\unitlength}{5.0ex}

\begin{figure}[htbp]
\begin{minipage}{0.49\textwidth}
\begin{center}
\begin{scriptsize}
\begin{small}
\begin{picture}(3,3)
\multiput(0.5,0.5)(0,1){3}{\multiput(0,0)(1,0){3}{\circle*{0.2}}}%
\put(0.75,0.75){\makebox(0,0){$1$}}
\put(1.75,1.75){\makebox(0,0){$2$}}
\put(2.75,2.75){\makebox(0,0){$3$}}
\drawline[1](2.5,1.5)(2.5,2.5)(1.5,2.5)
\drawline[1](1.5,0.5)(0.5,0.5)(0.5,1.5)
\drawline[1](2.5,0.5)(0.5,2.5)
\end{picture}
\end{small}
\end{scriptsize}
\caption{An idempotent and $\leq_3$-preserving operation}
\label{fig:2}
\end{center}
\end{minipage}
\hfill
\begin{minipage}{0.49\textwidth}
\begin{center}
\begin{scriptsize}
\begin{small}
\begin{picture}(3,3)
\multiput(0.5,0.5)(0,1){3}{\multiput(0,0)(1,0){3}{\circle*{0.2}}}%
\drawline[1](2.5,1.5)(2.5,2.5)(1.5,2.5)
\drawline[1](0.5,1.5)(1.5,1.5)(1.5,0.5)
\put(1.5,0.5){\oval(2,0.5)[b]}\put(0.5,1.5){\oval(0.5,2)[l]}
\put(0.75,0.75){\makebox(0,0){$1$}}
\put(1.75,1.75){\makebox(0,0){$2$}}
\put(2.75,2.75){\makebox(0,0){$3$}}
\end{picture}
\end{small}
\end{scriptsize}
\caption{A quasitrivial operation that is not $\leq_3$-preserving}
\label{fig:3}
\end{center}
\end{minipage}
\end{figure}

\begin{lemma}\label{lem:an}
Let $F\colon X_n^2 \to X_n$ be a quasitrivial and $\leq_n$-preserving operation and let $a \in X_n$. If $a$ is an annihilator of $F$, then $a \in \{1,n\}$.
\end{lemma}

\begin{proof}
We proceed by contradiction. Suppose that $a \in X_n \setminus \{1,n\}$. Since $F$ is quasitrivial, we have $F(1,n) \in \{1,n\}$. Suppose that $F(1,n) = 1 = F(1,1)$ (the other case is similar). Then
$$
1 = F(1,1) \leq_n F(1,a) \leq_n F(1,n) = 1,
$$
and hence $F(1,a)=1$ a contradiction.
\end{proof}

\begin{proposition}\label{thm:contour2}
Let $F \colon X_n^2 \to X_n$ be quasitrivial and $\leq_n$-preserving. Then $F$ is bisymmetric iff it satisfies \eqref{eq:deg} for some $k \in \{1,\ldots,n\}$.
\end{proposition}

\begin{proof}
(Necessity) This follows from Proposition \ref{thm:contour}.

(Sufficiency) We proceed by induction on $n$. The result clearly holds for $n=1$. Suppose that it holds for some $n \geq 1$ and let us show that it still holds for $n+1$. Assume that $F\colon X_{n+1}^2 \to X_{n+1}$ is quasitrivial, $\leq_{n+1}$-preserving, and satisfies
$$\deg_F = (\underbrace{k-1,\ldots,k-1}_{\textstyle{k~ \text{times}}},2k,2k+2,\ldots,2n)$$
for some $k \in \{1,\ldots,n+1\}$. If $k=n+1$ then by Proposition \ref{prop:pi} we have that $F=\pi_1$ or $F=\pi_2$ and hence $F$ is clearly bisymmetric. Otherwise, if $k \in \{1,\ldots,n\}$ then, by the form of the degree sequence of $F$, there exists an element $a \in X_{n+1}$ such that $\deg_F(a)=2n$. Using Proposition \ref{prop:an} we have that $a$ is an annihilator of $F$. Moreover, by Lemma \ref{lem:an}, we have $a \in \{1,n+1\}$. Suppose that $a=n+1$ (the other case is similar). Then, $F'=F|_{X_{n}^2}$ is clearly quasitrivial, $\leq_{n}$-preserving, and satisfies \eqref{eq:deg}. Thus, by induction hypothesis, $F'$ is bisymmetric. Since $a=n+1$ is the annihilator of $F$, we necessarily have that $F$ is bisymmetric.
\end{proof}

The following theorem, which is an immediate consequence of Theorem \ref{thm:B} and Propositions \ref{thm:bs} and \ref{thm:contour2}, provides characterizations of the class of bisymmetric, quasitrivial, and order-preserving operations.

\begin{theorem}\label{thm:BND}
Let $\leq$ be a linear ordering on $X$ and let $F \colon X^2 \to X$ be an operation. The following assertions are equivalent.
\begin{enumerate}
  \item[(i)] $F$ is bisymmetric, quasitrivial, and $\leq$-preserving.
  \item[(ii)] $F$ is of the form \eqref{eq:kimura} for some quasilinear weak ordering $\precsim$ on $X$ that is weakly single-peaked w.r.t.\ $\leq$.
  \item[(iii)] $F$ is of the form \eqref{eq:kimura} for some weak ordering $\precsim$ on $X$ that is weakly single-peaked w.r.t.\ $\leq$ and for any linear ordering $\preceq'$ subordinated to $\precsim$, $F$ is $\preceq'$-preserving.
  \item[(iv)] $F$ is of the form \eqref{eq:kimura} for some weak ordering $\precsim$ on $X$ that is weakly single-peaked w.r.t.\ $\leq$ and for any linear ordering $\preceq'$ subordinated to $\precsim$, $F$ has no $\preceq'$-disconnected level set.
\end{enumerate}
If $(X,\leq) = (X_n,\leq_n)$ for some integer $n \geq 1$, then any of the assertions (i)-(iv) is equivalent to the following one.
\begin{enumerate}
\item[(v)] $F$ is quasitrivial, $\leq_n$-preserving, and satisfies \eqref{eq:deg} for some $k \in \{1,\ldots,n\}$.
\end{enumerate}
\end{theorem}

\begin{remark}
\begin{enumerate}
\item[(a)] We observe that in \cite[Theorem 3.7]{CouDevMar2} the authors proved that an operation $F\colon X_n^2 \to X_n$ is associative, commutative, quasitrivial, and $\leq_n$-preserving iff it is quasitrivial, $\leq_n$-preserving, and satisfies
    $$\deg_F =(0,2,\ldots,2(n-1)).$$
    Surprisingly, in Theorem \ref{thm:BND}(v), we have obtained a similar result by relaxing commutativity into bisymmetry. Moreover, it provides an easy test to check whether a quasitrivial and $\leq_n$-preserving operation on $X_n$ is bisymmetric. Indeed, the quasitrivial F and $\leq_3$-preserving operation $F\colon X_3^2 \to X_3$ whose contour plot is depicted in Figure \ref{fig:kim} (left) is not bisymmetric since $\deg_F = (0,3,3)$ is not of the form given in Theorem \ref{thm:BND}(v). It is important to note that in this test, associativity of the given operation need not be checked.
\item[(b)] Recall that an operation $F\colon X^2 \to X$ is said to be \emph{autodistributive} \cite[Section 6.5]{Acz4} if it satisfies the following functional equations
    $$
    F(x,F(y,z))=F(F(x,y),F(x,z)), \qquad x,y,z \in X,
    $$
    and
    $$
    F(F(x,y),z)=F(F(x,z),F(y,z)), \qquad x,y,z \in X.
    $$
    We observe that, under quasitriviality, bisymmetry is equivalent to autodistributivity. Indeed, it is known \cite[Corollary 7.7]{Jez78} that any autodistributive and quasitrivial operation is bisymmetric. Conversely, it was observed in \cite[p.\ 39]{GMRE09} that any bisymmetric and idempotent operation is autodistributive. Thus, in Theorems \ref{thm:B}(i) and \ref{thm:BND}(i) we can replace bisymmetry with autodistributivity.
\end{enumerate}
\end{remark}


\section{Enumerations of bisymmetric and quasitrivial operations}

In \cite[Section 4]{CouDevMar2}, the authors enumerated the class of associative and quasitrivial operations $F\colon X_n^2 \to X_n$ as well as some subclasses obtained by considering commutativity and order-preservation. Some of these computations gave rise to previously unknown integer sequences, which were then posted in OEIS (for instance A292932 and A293005). In the same spirit, this section is devoted to the enumeration of the class of bisymmetric and quasitrivial operations $F \colon X_n^2 \to X_n$ as well as some of its subclasses. The integer sequences that emerge from our investigation are also now posted in OEIS (see, e.g., A296943).

We also consider either the (ordinary) generating function (GF) or the exponential generating function (EGF) of a given sequence $(s_n)_{n\geq 0}$. Recall that when these functions exist, they are respectively defined by the power series
$$
S(z) ~=~ \sum_{n\geq 0}s_n{\,}z^n\quad\text{and}\quad \hat S(z) ~=~ \sum_{n\geq 0}s_n{\,}\frac{z^n}{n!}{\,}.
$$

For any integer $n \geq 0$ we denote by $p(n)$ the number of weak orderings on $X_n$ that are quasilinear. We also denote by $p_e(n)$ (resp.\ $p_a(n)$) the number of weak orderings $\precsim$ on $X_n$ that are quasilinear and for which $X_n$ has exactly one minimal element (resp.\ exactly one maximal element) for $\precsim$. By convention, we set $p(0)=p_e(0)=p_a(0)=0$. Clearly, $p_e(n)$ is the number of linear orderings on $X_n$, namely $p_e(n)=n!$ for $n \geq 1$. Proposition \ref{prop:uw1} below provides explicit formulas for $p(n)$ and $p_a(n)$. The first few values of these sequences are shown in Table \ref{tab:p}.\footnote{Note that the sequence A000142 differs from $(p_e(n))_{n\geq 0}$ only at $n=0$.}

\begin{proposition}\label{prop:uw1}
The sequence $(p(n))_{n\geq 0}$ satisfies the linear recurrence equation
$$
p(n+1) - (n+1)p(n) = 1, \qquad n \geq 0,
$$
with $p(0)=0$, and we have the closed-form expression
$$
p(n) = n!\sum_{k=1}^{n}\frac{1}{k!}, \qquad n \geq 1.
$$
Moreover, its EGF is given by $\hat P(z) = (e^z-1)/(1-z)$. Furthermore, for any integer $n \geq 2$ we have $p_a(n)=p(n)-1$, 
with $p_a(0)=0$ and $p_a(1)=1$.
\end{proposition}

\begin{proof}
We clearly have
$$
p(n) = \sum_{k=1}^n{n\choose k,1,\ldots,1},\quad n\geq 1,
$$
where the multinomial coefficient ${n\choose k,1,\ldots,1}$ provides the number of ways to put the elements $1,\ldots,n$ into $(n-k+1)$ classes of sizes $k,1,\ldots,1$. The claimed linear recurrence equation and the EGF of $(p(n))_{n \geq 1}$ follow straightforwardly. Regarding the sequence $(p_a(n))_{n \geq 0}$ we observe that $\max_{\precsim}X_n \neq X_n$ whenever $n \geq 2$ (see Fact \ref{lem:count}).
\end{proof}

\begin{table}[htbp]
$$
\begin{array}{|c|rrr|}
\hline n & p(n) & p_e(n) & p_a(n)  \\
\hline 0 & 0 & 0 & 0 \\
1 & 1 & 1 & 1\\
2 & 3 & 2 & 2\\
3 & 10 & 6 & 9\\
4 & 41 & 24 & 40\\
5 & 206 & 120 & 205\\
6 & 1237 & 720 & 1236\\
\hline
\mathrm{OEIS}^{\mathstrut} & \mathrm{A002627} & \mathrm{A000142} & \mathrm{A296964} \\
\hline
\end{array}
$$
\caption[]{First few values of $p(n)$, $p_e(n)$, and $p_a(n)$}
\label{tab:p}
\end{table}

For any integer $n \geq 0$ we denote by $q(n)$ the number of bisymmetric and quasitrivial operations $F\colon X_n^2 \to X_n$. We also denote by $q_e(n)$ (resp.\ $q_a(n)$) the number of bisymmetric and quasitrivial operations $F\colon X_n^2 \to X_n$ that have neutral elements (resp.\ annihilator elements). By convention, we set $q(0)=q_e(0)=q_a(0)=0$. Proposition \ref{prop:qms} provides explicit formulas for these sequences. The first few values of these sequences are shown in Table \ref{tab:q}.\footnote{Note that the sequence A000142 differs from $(q_e(n))_{n\geq 0}$ only at $n=0$.}

\begin{proposition}\label{prop:qms}
The sequence $(q(n))_{n\geq 0}$ satisfies the linear recurrence equation
$$
q(n+1) - (n+1)q(n) = 2, \qquad n \geq 1,
$$
with $q(0)=0$ and $q(1)=1$, and we have the closed-form expression
$$
q(n) = 2p(n)-n! = n!\left(2\sum_{i=1}^{n}\frac{1}{i!}-1\right), \qquad n \geq 1.
$$
Moreover, its EGF is given by $\hat Q(z) = (2e^z-3)/(1-z)$. Furthermore, for any integer $n \geq 1$ we have $q_{e}(n)=n!$, with $q_{e}(0)=0$. Also, for any integer $n \geq 2$ we have $q_a(n)=q(n)-2$, with $q_a(0)=0$ and $q_a(1)=1$. 
\end{proposition}

\begin{proof}
It is not difficult to see that the number of bisymmetric and quasitrivial operations on $X_n$ is given by
$$
q(n) = 2p(n) - n! =  n!\left(2\sum_{i=1}^{n}\frac{1}{i!}-1\right), \qquad n \geq 1.
$$
Indeed, since $F \colon X_n^2 \to X_n$ is bisymmetric and quasitrivial, we have by Theorem \ref{thm:B} that $F$ is of the form \eqref{eq:kimura} for some quasilinear weak ordering $\precsim$ on $X_n$. Since $F|_{\min_{\precsim}X_n} = \pi_1 ~\mbox{or}~ \pi_2$, we have to count twice the number of $k$-element subsets of $X_n$, for every $k \in \{1,\ldots,n\}$. However, the number of linear orderings on $X_n$ should be counted only once (indeed, by Remark \ref{rem:sym}(a) there is a one-to-one correspondence between linear orderings and bisymmetric, commutative, and quasitrivial operations on $X_n$). Hence, $q(n) = 2p(n) - n!$. The claimed linear recurrence equation and the EGF of $(q(n))_{n \geq 1}$ follow straightforwardly. Using Corollary \ref{cor:equiv}(ii) and Remark \ref{rem:sym}(a), we observe that the sequence $(q_e(n))_{n \geq 0}$, with $q_{e}(0)=0$, gives the number of linear orderings on $X_n$. Finally, regarding the sequence $(q_a(n))_{n \geq 0}$, we observe that $\max_{\preceq}X_n \neq X_n$ whenever $n \geq 2$ (see Fact \ref{lem:count}).
\end{proof}

\begin{table}[htbp]
$$
\begin{array}{|c|rrr|}
\hline n & q(n) & q_e(n) & q_a(n)  \\
\hline 0 & 0 & 0 & 0 \\
1 & 1 & 1 & 1 \\
2 & 4 & 2 & 2\\
3 & 14 & 6 & 12 \\
4 & 58 & 24 & 56\\
5 & 292 & 120 & 290\\
6 & 1754 & 720 & 1752\\
\hline
\mathrm{OEIS}^{\mathstrut} & \mathrm{A296943} & \mathrm{A000142} & \mathrm{A296944}\\
\hline
\end{array}
$$
\caption[]{First few values of $q(n)$, $q_e(n)$, and $q_a(n)$}
\label{tab:q}
\end{table}

We are now interested in computing the size of the class of bisymmetric, quasitrivial, and $\leq_n$-preserving operations $F\colon X_n^2 \to X_n$. To this extent, we first consider a preliminary result.

%

\begin{lemma}\label{lem:str}
Let $\precsim$ be a quasilinear weak ordering on $X_n$ that is weakly single-peaked w.r.t.\ $\leq_n$. If $\max_{\precsim}X_n \neq X_n$, then $\max_{\precsim}X_n \subseteq \{1,n\}$ and $|\max_{\precsim}X_n|=1$.
\end{lemma}

\begin{proof}
We proceed by contradiction. By Fact \ref{lem:count}, the set $\max_{\precsim}X_n$ contains exactly one element. Suppose that $\max_{\precsim}X_n = \{x\}$, where, $x \in X_n\setminus \{1,n\}$. Then the triplet $(1,x,n)$ violates weak single-peakedness of $\precsim$.
\end{proof}

For any integer $n \geq 0$ we denote by $u(n)$ the number of quasilinear weak orderings $\precsim$ on $X_n$ that are weakly single-peaked w.r.t.\ $\leq_n$. We also denote by $u_e(n)$ (resp.\ $u_a(n)$) the number of quasilinear weak orderings $\precsim$ on $X_n$ that are weakly single-peaked w.r.t.\ $\leq_n$ and for which $X_n$ has exactly one minimal element (resp.\ exactly one maximal element) for $\precsim$. By convention, we set $u(0)=u_e(0)=u_a(0)=0$. Proposition \ref{prop:uw2} provides explicit formulas for these sequences. The first few values of these sequences are shown in Table \ref{tab:u}.

\begin{proposition}\label{prop:uw2}
The sequence $(u(n))_{n\geq 0}$ satisfies the linear recurrence equation
$$
u(n+1) = 2u(n) + 1, \qquad n \geq 0,
$$
with $u(0)=0$, and we have the closed-form expression
$$
u(n) = 2^n-1, \qquad n \geq 0.
$$
Moreover, its GF is given by $U(z) = z/(2z^2-3z+1)$. Furthermore, for any integer $n \geq 1$ we have $u_e(n)=2^{n-1}$ with $u_e(0)=0$. Also, for any integer 
$n \geq 2$ we have $u_a(n)=u(n)-1$ with $u_a(0)=0$ and $u_a(1)=1$.
\end{proposition}

\begin{proof}
We clearly have $u(0)=0$ and $u(1)=1$. So let us assume that $n \geq 2$. If $\precsim$ is a quasilinear weak ordering on $X_n$ that is weakly single-peaked w.r.t.\ $\leq_n$, then by Lemma \ref{lem:str}, either $\max_{\precsim}X_n = X_n$ or $\max_{\precsim}X_n = \{1\}$ or $\max_{\precsim}X_n = \{n\}$. In the two latter cases, it is clear that the restriction of $\precsim$ to $X_n \setminus \max_{\precsim}X_n$ is quasilinear and weakly single-peaked w.r.t.\ the restriction of $\leq_n$ to $X_n \setminus \max_{\precsim}X_n$. It follows that the number $u(n)$ of quasilinear weak orderings on $X_n$ that are weakly single-peaked w.r.t.\ $\leq_n$ satisfies the first order linear equation
$$
u(n) = 1 + u(n-1) + u(n-1), \qquad n \geq 2.
$$
The stated expression of $u(n)$ and the GF of $(u(n))_{n \geq 2}$ follow straightforwardly. Using Fact \ref{lem:min} and Remark \ref{def:Black}, we observe that the sequence $(u_e(n))_{n \geq 0}$, with $u_e(0)=0$, gives the number of linear orderings on $X_n$ that are single-peaked w.r.t.\ $\leq_n$. Finally, regarding the sequence $(u_a(n))_{n \geq 0}$, we observe that $\max_{\precsim}X_n \neq X_n$ whenever $n \geq 2$ (see Fact \ref{lem:count}).
\end{proof}

\begin{table}[htbp]
$$
\begin{array}{|c|rrr|}
\hline n & u(n) & u_e(n) & u_a(n) \\
\hline 0 & 0 & 0 & 0\\
1 & 1 & 1 & 1\\
2 & 3 & 2 & 2\\
3 & 7 & 4 & 6\\
4 & 15 & 8 & 14\\
5 & 31 & 16 & 30\\
6 & 63 & 32 & 62\\
\hline
\mathrm{OEIS}^{\mathstrut} & \mathrm{A000225} & \mathrm{A131577} & \mathrm{A296965} \\
\hline
\end{array}
$$
\caption[]{First few values of $u(n)$, $u_e(n)$, and $u_a(n)$}
\label{tab:u}
\end{table}

\begin{example}
The $p(3)=10$ quasilinear weak orderings on $X_3$ are: $1 \prec 2 \prec 3$, $1 \prec 3 \prec 2$, $2 \prec 1 \prec 3$, $2 \prec 3 \prec 1$, $3 \prec 1 \prec 2$, $3 \prec 2 \prec 1$, $1 \sim 2 \prec 3$, $1 \sim 3 \prec 2$, $2 \sim 3 \prec 1,$ and $1 \sim 2 \sim 3$. Among these quasilinear weak orderings, $p_e(3)=6$ have exactly one minimal element, $p_a(3)=9$ have exactly one maximal element, and $u(3)=7$ are weakly single-peaked w.r.t. $\leq_3$.
\end{example}

For any integer $n \geq 0$ we denote by $v(n)$ the number of bisymmetric, quasitrivial, and $\leq_n$-preserving operations $F \colon X_n^2 \to X_n$. We also denote by $v_e(n)$ (resp.\ $v_a(n)$) the number of bisymmetric, quasitrivial, and $\leq_n$-preserving operations $F \colon X_n^2 \to X_n$ that have neutral elements (resp.\ annihilator elements). By convention, we set $v(0)=v_e(0)=v_a(0)=0$. Proposition \ref{prop:count} provides explicit formulas for these sequences. The first few values of these sequences are shown in Table \ref{tab:v}.

\begin{proposition}\label{prop:count}
The sequence $(v(n))_{n\geq 0}$ satisfies the linear recurrence equation
$$
v(n+1) = 2v(n) + 2, \qquad n \geq 1,
$$
with $v(0)=0$ and $v(1)=1$, and we have the closed-form expression
$$
v(n) = 3 \cdot 2^{n-1}-2, \qquad n \geq 1.
$$
Moreover, its GF is given by $V(z) = z(z+1)/(2z^2-3z+1)$. Furthermore, for any integer $n \geq 1$ we have $v_e(n)=2^{n-1}$ with $v_e(0)=0$. Also, for any integer 
$n \geq 2$ we have $v_a(n)=v(n)-2$ with $v_a(0)=0$ and $v_a(1)=1$.
\end{proposition}

\begin{proof}
We clearly have $v(0)=0$ and $v(1)=1$. So let us assume that $n \geq 2$. If $F \colon X_n^2 \to X_n$ is a bisymmetric, quasitrivial, and $\leq_n$-preserving operation, then by Theorem \ref{thm:BND} it is of the form \eqref{eq:kimura} for some quasilinear weak ordering $\precsim$ on $X_n$ that is weakly single-peaked w.r.t.\ $\leq_n$. By Lemma \ref{lem:str}, either $\max_{\precsim}X_n = X_n$ or $\max_{\precsim}X_n = \{1\}$ or $\max_{\precsim}X_n = \{n\}$. In the first case, we have to consider the two projections $F = \pi_1$ and $F = \pi_2$. In the two latter cases, it is clear that the restriction of $F$ to $(X_n \setminus \max_{\precsim}X_n)^2$ is still bisymmetric, quasitrivial, and $\leq_n'$-preserving, where $\leq_n'$ is the restriction of $\leq_n$ to $X_n \setminus \max_{\precsim}X_n$. It follows that the number $v(n)$ of quasitrivial, bisymmetric, and $\leq_n$-preserving operations $F \colon X_n^2 \to X_n$ satisfies the first order linear equation
$$
v(n) = 2 + v(n-1) + v(n-1), \qquad n \geq 2.
$$
The stated expression of $v(n)$ and the GF of $(v(n))_{n \geq 2}$ follow straightforwardly. Using Corollary \ref{cor:equiv}(ii) and Remark \ref{rem:sym}(b), we observe that the sequence $(v_e(n))_{n \geq 0}$, with $v_e(0)=0$, gives the number of linear orderings on $X_n$ that are single-peaked w.r.t.\ $\leq_n$. Finally, regarding the sequence $(v_a(n))_{n \geq 0}$, we observe that $\max_{\precsim}X_n \neq X_n$ whenever $n \geq 2$ (see Fact \ref{lem:count}). 
\end{proof}

\begin{remark}
We observe that an alternative characterization of the class of bisymmetric, quasitrivial, and $\leq_n$-preserving operations $F\colon X_n^2 \to X_n$ was obtained in \cite{Kiss}. Also, the explicit expression of $v(n)$ as stated in Proposition \ref{prop:count} was independently obtained in \cite{Kiss} by means of a totally different approach.
\end{remark}

\begin{table}[htbp]
$$
\begin{array}{|c|rrr|}
\hline n & v(n) & v_e(n) & v_a(n) \\
\hline 0 & 0 & 0 & 0\\
1 & 1 & 1 & 1\\
2 & 4 & 2 & 2\\
3 & 10 & 4 & 8\\
4 & 22 & 8 & 20\\
5 & 46 & 16 & 44\\
6 & 94 & 32 & 92\\
\hline
\mathrm{OEIS}^{\mathstrut} & \mathrm{A296953} & \mathrm{A131577} & \mathrm{A296954}\\
\hline
\end{array}
$$
\caption[]{First few values of $v(n)$, $v_e(n)$, and $v_a(n)$}
\label{tab:v}
\end{table}

\begin{example}
We show in Figure~\ref{fig:eos4} the $q(3)=14$ bisymmetric and quasitrivial operations on $X_3$. Among these operations, $q_e(3)=6$ have neutral elements, $q_a(3)=12$ have annihilator elements, and $v(3)=10$ are $\leq_3$-preserving.
\end{example}

\setlength{\unitlength}{3.5ex}
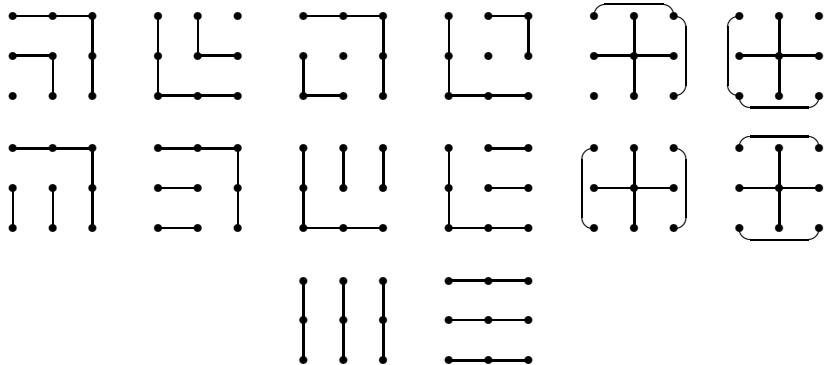
\begin{figure}[htbp]
\begin{center}
\begin{small}
\begin{picture}(3,3)
\multiput(0.5,0.5)(0,1){3}{\multiput(0,0)(1,0){3}{\circle*{0.2}}}
\drawline[1](0.5,2.5)(2.5,2.5)(2.5,0.5)\drawline[1](0.5,1.5)(1.5,1.5)(1.5,0.5)
\end{picture}
\hspace{0.01\textwidth}
\begin{picture}(3,3)
\multiput(0.5,0.5)(0,1){3}{\multiput(0,0)(1,0){3}{\circle*{0.2}}}
\drawline[1](0.5,2.5)(0.5,0.5)(2.5,0.5)\drawline[1](1.5,2.5)(1.5,1.5)(2.5,1.5)
\end{picture}
\hspace{0.01\textwidth}
\begin{picture}(3,3)
\multiput(0.5,0.5)(0,1){3}{\multiput(0,0)(1,0){3}{\circle*{0.2}}}
\drawline[1](0.5,2.5)(2.5,2.5)(2.5,0.5)\drawline[1](0.5,1.5)(0.5,0.5)(1.5,0.5)
\end{picture}
\hspace{0.01\textwidth}
\begin{picture}(3,3)
\multiput(0.5,0.5)(0,1){3}{\multiput(0,0)(1,0){3}{\circle*{0.2}}}
\drawline[1](0.5,2.5)(0.5,0.5)(2.5,0.5)\drawline[1](1.5,2.5)(2.5,2.5)(2.5,1.5)
\end{picture}
\hspace{0.01\textwidth}
\begin{picture}(3,3)
\multiput(0.5,0.5)(0,1){3}{\multiput(0,0)(1,0){3}{\circle*{0.2}}}
\drawline[1](0.5,1.5)(2.5,1.5)\drawline[1](1.5,0.5)(1.5,2.5)
\put(2.5,1.5){\oval(0.6,2)[r]}\put(1.5,2.5){\oval(2,0.6)[t]}
\end{picture}
\hspace{0.01\textwidth}
\begin{picture}(3,3)
\multiput(0.5,0.5)(0,1){3}{\multiput(0,0)(1,0){3}{\circle*{0.2}}}
\drawline[1](0.5,1.5)(2.5,1.5)\drawline[1](1.5,0.5)(1.5,2.5)
\put(0.5,1.5){\oval(0.6,2)[l]}\put(1.5,0.5){\oval(2,0.6)[b]}
\end{picture}

\medskip

\begin{picture}(3,3)
\multiput(0.5,0.5)(0,1){3}{\multiput(0,0)(1,0){3}{\circle*{0.2}}}
\drawline[1](0.5,2.5)(2.5,2.5)(2.5,0.5)\drawline[1](0.5,1.5)(0.5,0.5)\drawline[1](1.5,1.5)(1.5,0.5)
\end{picture}
\hspace{0.01\textwidth}
\begin{picture}(3,3)
\multiput(0.5,0.5)(0,1){3}{\multiput(0,0)(1,0){3}{\circle*{0.2}}}
\drawline[1](0.5,2.5)(2.5,2.5)(2.5,0.5)\drawline[1](0.5,1.5)(1.5,1.5)\drawline[1](0.5,0.5)(1.5,0.5)
\end{picture}
\hspace{0.01\textwidth}
\begin{picture}(3,3)
\multiput(0.5,0.5)(0,1){3}{\multiput(0,0)(1,0){3}{\circle*{0.2}}}
\drawline[1](0.5,2.5)(0.5,0.5)(2.5,0.5)\drawline[1](1.5,2.5)(1.5,1.5)\drawline[1](2.5,2.5)(2.5,1.5)
\end{picture}
\hspace{0.01\textwidth}
\begin{picture}(3,3)
\multiput(0.5,0.5)(0,1){3}{\multiput(0,0)(1,0){3}{\circle*{0.2}}}
\drawline[1](0.5,2.5)(0.5,0.5)(2.5,0.5)\drawline[1](1.5,2.5)(2.5,2.5)\drawline[1](1.5,1.5)(2.5,1.5)
\end{picture}
\hspace{0.01\textwidth}
\begin{picture}(3,3)
\multiput(0.5,0.5)(0,1){3}{\multiput(0,0)(1,0){3}{\circle*{0.2}}}
\drawline[1](0.5,1.5)(2.5,1.5)\drawline[1](1.5,0.5)(1.5,2.5)
\put(2.5,1.5){\oval(0.6,2)[r]}\put(0.5,1.5){\oval(0.6,2)[l]}
\end{picture}
\hspace{0.01\textwidth}
\begin{picture}(3,3)
\multiput(0.5,0.5)(0,1){3}{\multiput(0,0)(1,0){3}{\circle*{0.2}}}
\drawline[1](0.5,1.5)(2.5,1.5)\drawline[1](1.5,0.5)(1.5,2.5)
\put(1.5,0.5){\oval(2,0.6)[b]}\put(1.5,2.5){\oval(2,0.6)[t]}
\end{picture}

\medskip

\begin{picture}(3,3)
\multiput(0.5,0.5)(0,1){3}{\multiput(0,0)(1,0){3}{\circle*{0.2}}}
\drawline[1](0.5,0.5)(0.5,2.5)\drawline[1](1.5,0.5)(1.5,2.5)\drawline[1](2.5,0.5)(2.5,2.5)
\end{picture}
\hspace{0.01\textwidth}
\begin{picture}(3,3)
\multiput(0.5,0.5)(0,1){3}{\multiput(0,0)(1,0){3}{\circle*{0.2}}}
\drawline[1](0.5,2.5)(2.5,2.5)\drawline[1](0.5,1.5)(2.5,1.5)\drawline[1](0.5,0.5)(2.5,0.5)
\end{picture}

\end{small}
\caption{The 14 bisymmetric and quasitrivial operations on $X_3$}
\label{fig:eos4}
\end{center}
\end{figure}

\section{Quasilinearity and weak single-peakedness}

In this section we investigate some properties of the quasilinear weak orderings on $X$ that are weakly single-peaked w.r.t.\ a fixed linear ordering $\leq$ on $X$.

The following lemma provides a characterization of quasilinearity under weak single-peakedness.

\begin{lemma}\label{lem:med}
Let $\leq$ be a linear ordering on $X$ and let $\precsim$ be a weak ordering on $X$ that is weakly single-peaked w.r.t.\ $\leq$. Then $\precsim$ is quasilinear iff there exist no $a,b,c\in X$, with $b \in \mathrm{conv}_{\leq}(a,c)$, such that $b \prec a \sim c$.
\end{lemma}

\begin{proof}
(Necessity) Obvious.

(Sufficiency) We proceed by contradiction. Suppose that there exist pairwise distinct $a,b,c \in X$, such that $a \prec b \sim c$. By weak single-peakedness we have $a \in \mathrm{conv}_{\leq}(b,c)$, a contradiction.
\end{proof}

The following lemma provides a characterization of weak single-peakedness under quasilinearity.

\begin{lemma}\label{lem:med2}
Let $\leq$ be a linear ordering on $X$ and let $\precsim$ be a weak ordering on $X$ that is quasilinear. Then $\precsim$ is weakly single-peaked w.r.t.\ $\leq$ iff for any $a,b,c \in X$ such that $b \in \mathrm{conv}_{\leq}(a,c)$ we have $b \precsim a$ or $b \precsim c$.
\end{lemma}

\begin{proof}
(Necessity) Obvious.

(Sufficiency) We proceed by contradiction. Suppose that there exist $a,b,c \in X$, with $b \in \mathrm{conv}_{\leq}(a,c)$, such that $b \succ a$ and $b \succsim c$ (the case $b \succsim a$ and $b \succ c$ is similar). We only have two cases to consider. If $b \sim c$, then quasilinearity is violated. Otherwise, if $b \succ c$, we also arrive at a contradiction.
\end{proof}


From Lemmas \ref{lem:med} and \ref{lem:med2} we obtain the following characterization.

\begin{proposition}\label{prop:ch}
Let $\leq$ be a linear ordering on $X$ and let $\precsim$ be a weak ordering on $X$. Let us also consider the following assertions
\begin{enumerate}
\item[(i)] $\precsim$ is weakly single-peaked w.r.t.\ $\leq$.
\item[(ii)] There exist no $a,b,c \in X$, with $b \in \mathrm{conv}_{\leq}(a,c)$, such that $b \prec a \sim c$.
\item[(iii)] For any $a,b,c \in X$ such that $b \in \mathrm{conv}_{\leq}(a,c)$ we have $b \precsim a$ or $b \precsim c$.
\item[(iv)] $\precsim$ is quasilinear.
\end{enumerate}
Then the conjunction of assertions (i) and (ii) holds iff the conjunction of assertions (iii) and (iv) holds.
\end{proposition}

%
The following theorem provides a characterization of weak single-peakedness.

\begin{definition}[{see \cite[Definition 5.4]{CouDevMar2}}]
Let $\leq$ be a linear ordering on $X$ and let $\precsim$ be a weak ordering on $X$. A subset $P \subseteq X$, $|P| \geq 2$, is called a \emph{plateau with respect to $(\leq,\precsim)$} if $P$ is convex with respect to $\leq$ and if there exists $x\in X$ such that $P \subseteq [x]_{\sim}$. Moreover, the plateau $P$ is said to be \emph{$\precsim$-minimal} if for all $a \in X$ verifying $a \precsim P$ there exists $z \in P$ such that $a\sim z$.
\end{definition}

\begin{theorem}[{see \cite[Theorem 5.6]{CouDevMar2}}]\label{thm:graph1}
Let $\leq$ be a linear ordering on $X$ and let $\precsim$ be a weak ordering on $X$. Then $\precsim$ is weakly single-peaked w.r.t. $\leq$ iff the following conditions hold.
\begin{enumerate}
\item[(i)] For any $a,b,c \in X$ such that $b \in \mathrm{conv}_{\leq}(a,c)$, we have $b \precsim a$ or $b \precsim c$.
\item[(ii)] If $P \subseteq X$, $|P| \geq 2$, is a plateau w.r.t.\ $(\leq,\precsim)$, then it is $\precsim$-minimal.
\end{enumerate}
\end{theorem}

\begin{remark}
In \cite[Remark 7]{CouDevMar2} the authors observed that the weak orderings on $X_n$ that are weakly single-peaked w.r.t.\ $\leq_n$ are exactly the weak orderings on $X_n$ that are \emph{single-plateaued w.r.t.\ $\leq_n$} (see, e.g., \cite[Definition 4 and Lemma 17]{Fi15}).
\end{remark}

In \cite[Section 5]{CouDevMar2}, it was observed that the weak single-peakedness property of a weak ordering $\precsim$ on $X$ w.r.t.\ some linear ordering $\leq$ on $X$ can be checked by plotting a function, say $f_{\precsim}$, in a rectangular coordinate system in the following way. Represent the linearly ordered set $(X,\leq)$ on the horizontal axis and the reversed version of the weakly ordered set $(X,\precsim)$ on the vertical axis\footnotemark. The function $f_{\precsim}$ is then defined by its graph $\{(x,x)\mid x \in X\}$. Condition (i) of Theorem \ref{thm:graph1} simply says that the graph of $f_{\precsim}$ is V-free, i.e., there exist no three points $(a,a)$, $(b,b)$, $(c,c)$ in V-shape. Condition (ii) of Theorem \ref{thm:graph1} says that the graph of $f_{\precsim}$ is reversed L-free and L-free, i.e., the two patterns shown in Figure \ref{fig:fpat}, where each horizontal part is a plateau, cannot occur.
\footnotetext{In this representation, two equivalent elements of X have the same position on the vertical axis; see, e.g.,
Figure \ref{fig:wsp}.}

Thus, according to Theorem \ref{thm:graph1}, a weak ordering $\precsim$ on $X$ is weakly single-peaked w.r.t.\ some fixed linear ordering $\leq$ on $X$ iff the graph of $f_{\precsim}$ is V-free, L-free, and reversed L-free.

\setlength{\unitlength}{4ex}
\begin{figure}[htbp]
\begin{center}
\begin{small}
\begin{picture}(6.5,2.5)
\drawline[1](0.5,0.5)(2,0.5)(2.5,2)\drawline[1](4,2)(4.5,0.5)(6,0.5)
\put(0.5,0.5){\circle*{0.15}}\put(2,0.5){\circle*{0.15}}\put(2.5,2){\circle*{0.15}}
\put(4,2){\circle*{0.15}}\put(4.5,0.5){\circle*{0.15}}\put(6,0.5){\circle*{0.15}}
\end{picture}
\end{small}
\caption{The two patterns excluded by condition (ii)}
\label{fig:fpat}
\end{center}
\end{figure}
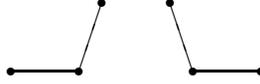

\begin{example}\label{ex:wsp}
Figure~\ref{fig:wsp} gives the functions $f_{\precsim}$ and $f_{\precsim'}$ corresponding to the weak orderings $2\sim 3\prec 1\sim 4$ and $1\prec' 2\sim' 3\prec' 4$, respectively, on $X_4$. We see that $\precsim$ is weakly single-peaked w.r.t.\ $\leq_4$ since the graph of $f_{\precsim}$ is V-free, L-free, and reversed L-free while $\precsim'$ is not weakly single-peaked w.r.t.\ $\leq_4$ since the graph of $f_{\precsim'}$ is not L-free. It has the plateau $P=\{2,3\}$, which is not $\precsim'$-minimal.
\end{example}

\setlength{\unitlength}{3.5ex}
\begin{figure}[htbp]
\begin{center}
\begin{small}
\null\hspace{0.03\textwidth}
\begin{picture}(6,5)
\put(0.5,0.5){\vector(1,0){5}}\put(0.5,0.5){\vector(0,1){4}}
\multiput(1.5,0.45)(1,0){4}{\line(0,1){0.1}}%
\multiput(0.45,1.5)(0,2){2}{\line(1,0){0.1}}%
\put(1.5,0){\makebox(0,0){$1$}}\put(2.5,0){\makebox(0,0){$2$}}\put(3.5,0){\makebox(0,0){$3$}}\put(4.5,0){\makebox(0,0){$4$}}
\put(-0.5,1.5){\makebox(0,0){$1\sim 4$}}\put(-0.5,3.5){\makebox(0,0){$2\sim 3$}}
\drawline[1](1.5,1.5)(2.5,3.5)(3.5,3.5)(4.5,1.5)
\put(1.5,1.5){\circle*{0.2}}\put(2.5,3.5){\circle*{0.2}}\put(3.5,3.5){\circle*{0.2}}\put(4.5,1.5){\circle*{0.2}}
\put(1.7,4){\makebox(0,0){\begin{normalsize}$f_{\precsim}$\end{normalsize}}}
\end{picture}
\hspace{0.1\textwidth}
\begin{picture}(6,5)
\put(0.5,0.5){\vector(1,0){5}}\put(0.5,0.5){\vector(0,1){4}}
\multiput(1.5,0.45)(1,0){4}{\line(0,1){0.1}}%
\multiput(0.45,1.5)(0,1){3}{\line(1,0){0.1}}%
\put(1.5,0){\makebox(0,0){$1$}}\put(2.5,0){\makebox(0,0){$2$}}\put(3.5,0){\makebox(0,0){$3$}}\put(4.5,0){\makebox(0,0){$4$}}
\put(0,1.5){\makebox(0,0){$4$}}\put(-0.5,2.5){\makebox(0,0){$2\sim' 3$}}\put(0,3.5){\makebox(0,0){$1$}}
\drawline[1](1.5,3.5)(2.5,2.5)(3.5,2.5)(4.5,1.5)
\put(1.5,3.5){\circle*{0.2}}\put(2.5,2.5){\circle*{0.2}}\put(3.5,2.5){\circle*{0.2}}\put(4.5,1.5){\circle*{0.2}}
\put(2.5,4){\makebox(0,0){\begin{normalsize}$f_{\precsim'}$\end{normalsize}}}
\end{picture}
\end{small}
\caption{$\precsim$ is weakly single-peaked (left) while $\precsim'$ is not (right)}
\label{fig:wsp}
\end{center}
\end{figure}
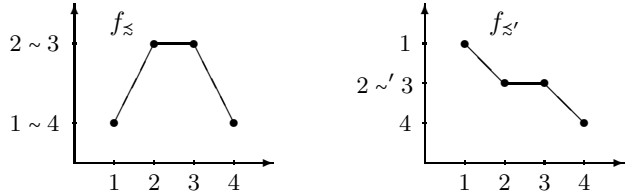

Let us now provide a characterization of quasilinear weak orderings $\precsim$ on $X$ that are weakly single-peaked w.r.t.\ some fixed linear ordering $\leq$ on $X$ in terms of the graph of $f_{\precsim}$.

Clearly, for any linear ordering $\leq$ on $X$ and any quasilinear weak ordering $\precsim$ on $X$, if a plateau w.r.t.\ $(\leq,\precsim)$ exists, then it is $\precsim$-minimal.

\begin{proposition}\label{prop:plat}
Let $\leq$ be a linear ordering on $X$ and let $\precsim$ be a weak ordering on $X$. Consider the assertions (iii) and (iv) of Proposition \ref{prop:ch} as well as the following one.
\begin{enumerate}
\item[(iv')] If there exist $a,b \in X$, $a \neq b$, such that $a \sim b$ then $\mathrm{conv}_{\leq}(a,b)$ is a plateau w.r.t.\ $(\leq,\precsim)$ and it is $\precsim$-minimal.
\end{enumerate}
Then we have ((iii) and (iv)) $\Rightarrow$ (iv'), and (iv') $\Rightarrow$ (iv).
\end{proposition}

\begin{proof}
$((iii) ~\mbox{and}~(iv)) \Rightarrow (iv')$. We proceed by contradiction. Let $a,b \in X$, $a \neq b$, such that $a \sim b$ and suppose that $\mathrm{conv}_{\leq}(a,b)$ is not a plateau w.r.t.\ $(\leq,\precsim)$. But then there exists $u \in \mathrm{conv}_{\leq}(a,b)$ such that either $u \prec a \sim b$, which contradicts (iv), or $u \succ a \sim b$ which contradicts (iii). Thus, $\mathrm{conv}_{\leq}(a,b)$ is a plateau w.r.t.\ $(\leq,\precsim)$ and it is $\precsim$-minimal by (iv).

$(iv') \Rightarrow (iv).$ We proceed by contradiction. Suppose that there exist pairwise distinct $a,b,c \in X$ such that $a \prec b \sim c$. But then $\mathrm{conv}_{\leq}(b,c)$ is a plateau w.r.t.\ $(\leq,\precsim)$ which is not $\precsim$-minimal, a contradiction to (iv').
\end{proof}

From Proposition \ref{prop:plat} it follows that conditions (iii) and (iv) hold iff conditions (iii) and (iv') hold. As previously explained, condition (iii) says that the graph of $f_{\precsim}$ is V-free. Condition (iv'), which is equivalent to quasilinearity under (iii), says that the graph of $f_{\precsim}$ is L-free, reversed L-free, and that any two elements which have the same position on the vertical axis form a plateau. Thus, combining Lemma \ref{lem:med2} and Proposition \ref{prop:plat}, we obtain the following result.

\begin{theorem}\label{thm:graph2}
Let $\leq$ be a linear ordering on $X$ and let $\precsim$ be a weak ordering on $X$. Then $\precsim$ is quasilinear and weakly single-peaked w.r.t. $\leq$ iff conditions (iii) and (iv') of Propositions \ref{prop:ch} and \ref{prop:plat} hold (i.e., the graph of $f_{\precsim}$ is V-free, L-free, reversed L-free, and any two elements which have the same position on the vertical axis form a plateau).
\end{theorem}

For instance, by Theorem \ref{thm:graph1}, the weak ordering $\precsim$ on $X_4$ defined in Example \ref{ex:wsp} is weakly single-peaked w.r.t.\ $\leq_4$ since the graph of $f_{\precsim}$ is V-free, L-free, and reversed L-free. However, by Theorem \ref{thm:graph2}, it is not quasilinear since the elements $1$ and $4$ have the same position on the vertical axis but do not form a plateau.

\begin{example}\label{ex:qlwsp}
Let us consider the operation $F\colon X_4^2\to X_4$ shown in Figure \ref{fig:qlwsp} (left). We can easily see that it is quasitrivial and $\leq_4$-preserving (see Proposition \ref{prop:jump}). Also, by Theorem \ref{thm:BND}, it is bisymmetric since $\deg_F=(1,1,4,6)$. Furthermore, it is of the form \eqref{eq:kimura} for the weak ordering $\precsim$ on $X_4$ obtained by ranking the numbers $\deg_F(x)$, $x\in X_4$, in nondecreasing order, that is, $2\sim 3\prec 1\prec 4$, see Figure~\ref{fig:qlwsp} (center). By Theorem \ref{thm:BND}, the weak ordering $\precsim$ is quasilinear and weakly single-peaked w.r.t.\ $\leq_4$. By Theorem~\ref{thm:graph2} the graph of $f_{\precsim}$ is V-free, L-free, reversed L-free, and any two elements which have the same position on the vertical axis form a plateau, see Figure~\ref{fig:qlwsp} (right).
\end{example}

\setlength{\unitlength}{4.5ex}
\begin{figure}[htbp]
\begin{center}
\begin{small}
\begin{picture}(4,4)
\multiput(0.5,0.5)(0,1){4}{\multiput(0,0)(1,0){4}{\circle*{0.2}}}
\drawline[1](0.5,3.5)(3.5,3.5)(3.5,0.5)\drawline[1](1.5,2.5)(1.5,1.5)\drawline[1](2.5,2.5)(2.5,1.5)
\drawline[1](0.5,2.5)(0.5,0.5)(2.5,0.5)
\put(0.5,0){\makebox(0,0){$1$}}\put(1,0){\makebox(0,0){$<$}}\put(1.5,0){\makebox(0,0){$2$}}
\put(2,0){\makebox(0,0){$<$}}\put(2.5,0){\makebox(0,0){$3$}}\put(3,0){\makebox(0,0){$<$}}
\put(3.5,0){\makebox(0,0){$4$}}
\end{picture}
\hspace{0.05\textwidth}
\begin{picture}(4,4)
\multiput(0.5,0.5)(0,1){4}{\multiput(0,0)(1,0){4}{\circle*{0.2}}}
\drawline[1](0.5,3.5)(3.5,3.5)(3.5,0.5)\drawline[1](0.5,2.5)(2.5,2.5)(2.5,0.5)\drawline[1](0.5,0.5)(0.5,1.5)
\drawline[1](1.5,0.5)(1.5,1.5)
\put(0.5,0){\makebox(0,0){$2$}}\put(1,0){\makebox(0,0){$\sim$}}\put(1.5,0){\makebox(0,0){$3$}}
\put(2,0){\makebox(0,0){$\prec$}}\put(2.5,0){\makebox(0,0){$1$}}\put(3,0){\makebox(0,0){$\prec$}}
\put(3.5,0){\makebox(0,0){$4$}}
\end{picture}
\hspace{0.06\textwidth}
\begin{picture}(5,4)
\put(0.5,0.5){\vector(1,0){4}}\put(0.5,0.5){\vector(0,1){3}}
\multiput(1,0.45)(0.9,0){4}{\line(0,1){0.1}}%
\multiput(0.45,1)(0,0.9){3}{\line(1,0){0.1}}%
\put(1,0){\makebox(0,0){$1$}}\put(1.9,0){\makebox(0,0){$2$}}\put(2.8,0){\makebox(0,0){$3$}}\put(3.7,0){\makebox(0,0){$4$}}
\put(0,1){\makebox(0,0){$4$}}\put(0,1.9){\makebox(0,0){$1$}}\put(-0.4,2.8){\makebox(0,0){$2\sim 3$}}
\drawline[1](1,1.9)(1.9,2.8)\drawline[1](1.9,2.8)(2.8,2.8)\drawline[1](2.8,2.8)(3.7,1)
\put(1,1.9){\circle*{0.15}}\put(1.9,2.8){\circle*{0.15}}\put(2.8,2.8){\circle*{0.15}}\put(3.7,1){\circle*{0.15}}
\put(2,3.5){\makebox(0,0){\begin{normalsize}$f_{\precsim}$\end{normalsize}}}
\end{picture}
\end{small}
\caption{Example~\ref{ex:qlwsp}}
\label{fig:qlwsp}
\end{center}
\end{figure}

We can now state an alternative characterization of weak single-peakedness under quasilinearity.

\begin{fact}\label{fact:qwsp}
Let $\leq$ be a linear ordering on $X$ and $\precsim$ be a weak ordering on $X$. If $\precsim$ is weakly single-peaked w.r.t.\ $\leq$, then there exists a linear ordering $\preceq'$ on $X$ subordinated to $\precsim$ that is single-peaked w.r.t.\ $\leq$.
\end{fact}

\begin{remark}
The converse of Fact \ref{fact:qwsp} does not hold in general. Indeed, if we consider the weak ordering $\precsim$ on $X_3$ defined by $1\prec 2\sim 3$, then it is clearly not weakly single-peaked w.r.t.\ $\leq_3$ but it has a subordinated linear ordering $\preceq'$ that is single-peaked w.r.t.\ $\leq_3$, namely $1 \prec' 2 \prec' 3$.
\end{remark}

\begin{lemma}
Let $\leq$ be a linear ordering on $X$ and $\precsim$ be a quasilinear weak ordering on $X$. Then $\precsim$ is weakly single-peaked w.r.t.\ $\leq$ iff there exists a linear ordering $\preceq'$ on $X$ subordinated to $\precsim$ that is single-peaked w.r.t.\ $\leq$.
\end{lemma}

\begin{proof}
(Necessity) This follows from Fact \ref{fact:qwsp}.

(Sufficiency) We proceed by contradiction. Suppose that there exist $a,b,c\in X$, with $b \in \mathrm{conv}_{\leq}(a,c)$, such that $b \succ a$ and $b \succsim c$ (the case $b \succsim a$ and $b \succ c$ is similar). If $b \succ a$ and $b \succ c$, then we have $b \succ' a$ and $b \succ' c$, which contradicts single-peakedness of $\preceq'$ w.r.t. $\leq$. Otherwise, if $a \prec b \sim c$, then quasilinearity of $\precsim$ is violated.
\end{proof}
\section{Conclusion}

This paper is based on two known results : (1) a characterization of the class of associative and quasitrivial operations on $X$ (see Theorem \ref{thm:Kim}) and (2) the fact that the class of bisymmetric and quasitrivial operations on $X$ is a subclass of the latter one (see Lemma \ref{lemma:bis}(iii)). By introducing the concept of quasilinearity for weak orderings on $X$ (see Definition \ref{def:quasi}) we provided a characterization of the class of bisymmetric and quasitrivial operations on $X$ (see Theorem \ref{thm:B}). To characterize those operations that are order-preserving (see Theorem \ref{thm:BND}), we considered the concepts of weak single-peakedness (see Definition \ref{def:WBlack}) and quasilinearity, and provided a graphical characterization of the conjunction of these two concepts (see Theorem \ref{thm:graph2}). Surprisingly, when $X=X_n$, we also provided a characterization of the latter classes (see Theorems \ref{thm:B} and \ref{thm:BND}) in terms of the degree sequences, which provides an easy test to check whether a quasitrivial operation is bisymmetric. The latter characterizations provide an answer to an open question posed in \cite[Section 5, Question (b)]{CouDevMar}. Also, when $X=X_n$, we enumerated the class of bisymmetric and quasitrivial operations as well as the subclass of bisymmetric, quasitrivial, and order-preserving operations (see Propositions \ref{prop:qms} and \ref{prop:count}). Furthermore, we enumerated all the quasilinear weak orderings on $X_n$ as well as those that are weakly single-peaked w.r.t. the linear ordering $\leq_n$ (see Propositions \ref{prop:uw1} and \ref{prop:uw2}) . All the new sequences that arose from our results were posted in OEIS.

In view of these results, some questions arise and we list some of them below.
\begin{itemize}
\item We provided a partial answer to an open question posed in \cite[Section 6]{CouDevMar2}. Namely, we were able to generalize \cite[Theorems 3.3 and 3.7]{CouDevMar2} by relaxing commutativity into bisymmetry (see Theorems \ref{thm:B} and \ref{thm:BND}). It would be interesting to search for a generalization of Theorems \ref{thm:B} and \ref{thm:BND} by removing bisymmetry in assertion (i).
\item Generalize Theorems \ref{thm:B} and \ref{thm:BND} by relaxing quasitriviality into idempotency.
\item Generalize Theorems \ref{thm:B} and \ref{thm:BND} for $n$-variable operations, with $n \geq 3$.
\item The integer sequences A000142, A002627, A000225, and A131577 were previously known in OEIS to solve enumeration issues neither related to quasilinearity nor weak single-peakedness. It would be interesting to search for possible one-to-one correspondences between those problems and ours.
\end{itemize}

\section*{Acknowledgments}

The author would like to thank Miguel Couceiro, Jean-Luc Marichal, and Bruno Teheux for fruitful discussions and valuable remarks. This research is supported by the Internal Research Project R-AGR-0500 of the University of Luxembourg and by the Luxembourg National Research Fund R-AGR-3080.



\begin{thebibliography}{99}

\bibitem{Ack}
N.~L.~Ackerman.
\newblock A characterization of quasitrivial $n$-semigroups.
\newblock To appear in {\em Algebra Universalis}.

\bibitem{Acz1}
J.~Acz\'{e}l.
\newblock The notion of mean values.
\newblock {\em Norske Videnskabers Selskabs Forhandlinger}, 19:83--86, 1946.

\bibitem{Acz2}
J.~Acz\'{e}l.
\newblock On mean values.
\newblock {\em Bull. of the Amer. Math. Soc.}, 54:392--400, 1948.

\bibitem{Acz4}
J.~Acz\'{e}l.
\newblock {\em Lectures on functional equations and their applications}.
\newblock Mathematics in Science and Engineering, Vol. 19. Academic Press, New York, 1966.

\bibitem{Acz3}
J.~Acz\'{e}l and J.~Dhombres.
\newblock {\em Functional Equations in Several Variables}.
\newblock Encyclopedia of Mathematics and Its Applications, vol. 31, Cambridge University Press, Cambridge, UK, 1989.

\bibitem{BerPer06}
S.~Berg and T.~Perlinger.
\newblock Single-peaked compatible preference profiles: some combinatorial results.
\newblock {\em Social Choice and Welfare} 27(1):89--102, 2006.

\bibitem{Bla48}
D.~Black.
\newblock On the rationale of group decision-making.
\newblock {\em J Polit Economy}, 56(1):23--34, 1948

\bibitem{CouDevMar}
M.~Couceiro, J.~Devillet, and J.-L.~Marichal.
\newblock Characterizations of idempotent discrete uninorms.
\newblock {\em Fuzzy Sets and Syst.}, 334:60--72, 2018.

\bibitem{CouDevMar2}
M.~Couceiro, J.~Devillet, and J.-L.~Marichal.
\newblock Quasitrivial semigroups: characterizations and enumerations.
\newblock {\em Semigroup Forum}, Submitted for revision. arXiv:1709.09162.

\bibitem{DevKiMar17}
J.~Devillet, G.~Kiss, and J.-L.~Marichal.
\newblock Characterizations of quasitrivial commutative nondecreasing associative operations.
\newblock Submitted for publication. arXiv:1705.00719.

\bibitem{Fi15}
Z.~Fitzsimmons.
\newblock Single-peaked consistency for weak orders is easy.
\newblock In Proc. of the 15th Conf. on Theoretical Aspects of Rationality and Knowledge (TARK 2015), pages 127--140, June 2015. arXiv:1406.4829.

\bibitem{FoMa97}
J.~Fodor and J.-L.~Marichal.
\newblock On nonstrict means.
\newblock {\em Aeq. Math.}, 54(3):308-327, 1997.

\bibitem{GMRE09}
M.~Grabisch, J.-L.~Marichal, R.~Mesiar, and E.~Pap.
\newblock {\em Aggregation Functions}.
\newblock Encyclopedia of Mathematics and its Applications, vol. 127, Cambridge University Press, Cambridge, UK, 2009.

\bibitem{Jez78}
J.~Je\v{z}ek and T.~Kepka.
\newblock Quasitrivial and nearly quasitrivial distributive groupoids and semigroups.
\newblock {\em Acta Univ. Carolin. - Math. Phys.}, 19(2):25--44, 1978.

\bibitem{Jez83}
J.~Je\v{z}ek and T.~Kepka.
\newblock Equational theories of medial groupoids.
\newblock {\em Algebra Universalis}, 17:174--190, 1983.

\bibitem{Jez83(1)}
J.~Je\v{z}ek and T.~Kepka.
\newblock Medial groupoids.
\newblock {\em Rozpravy \v{C}eskoslovensk\'{e} Akad. V\v{e}d, \v{R}ada Mat. P\v{r}\'{i}rod. V\v{e}d}, 93:93 pp., 1983.

\bibitem{Jez94}
J.~Je\v{z}ek and T.~Kepka.
\newblock Selfdistributive groupoids of small orders.
\newblock {\em Czechoslovak Math. J.}, 47(3):463--468, 1994.

\bibitem{Kep81}
T.~Kepka.
\newblock Quasitrivial groupoids and balanced identities.
\newblock {\em Acta Univ. Carolin. - Math. Phys.}, 22(2):49--64, 1981.

\bibitem{Kim}
N.~Kimura.
\newblock The structure of idempotent semigroups. I.
\newblock {\em Pacific J. Math.}, 8:257--275, 1958.

\bibitem{Kiss}
G.~Kiss.
\newblock Visual characterization of associative quasitrivial nondecreasing functions on finite chains.
\newblock Preprint. arXiv:1709.07340.

\bibitem{Lan80}
H.~L\"anger.
\newblock The free algebra in the variety generated by quasi-trivial semigroups.
\newblock {\em Semigroup forum}, 20:151--156, 1980.

\bibitem{MasMonTor03}
M.~Mas, M.~Monserrat and J.~Torrens.
\newblock On bisymmetric operators on a finite chain.
\newblock {\em IEEE Trans. Fuzzy Systems}, 11:647--651, 2003.

\bibitem{Dave}
D.~McLean.
\newblock Idempotent semigroups.
\newblock {\em Amer. Math. Monthly}, 61:110--113, 1954.

\bibitem{San05}
W.~Sander.
\newblock Some aspects of functional equations.
\newblock In: E.~P.~Klement and R.~Mesiar (eds.) {\em Logical, Algebraic, Analytic and Probabilistic Aspects of Triangular Norms}, pp.\ 143--187. Elsevier, Amsterdam, 2005.

\bibitem{Slo}
N.~J.~A.~Sloane (editor).
\newblock The On-Line Encyclopedia of Integer Sequences.
\newblock http://www.oeis.org

\bibitem{SuLiuPed17}
Y.~Su, H.-W.~Liu, and W.~Pedrycz.
\newblock On the discrete bisymmetry.
\newblock {\em IEEE Trans. Fuzzy Systems}. To appear. doi:10.1109/TFUZZ.2016.2637376
\end{thebibliography}
\end{document}